\documentclass[12pt,a4paper]{amsart}

\usepackage{amsfonts}
\usepackage{esint} \usepackage{amsthm}
\usepackage{amsmath}
\usepackage{xcolor}
\usepackage{amssymb}
\usepackage{mathtools} \usepackage{amscd}
\usepackage{dsfont} \usepackage[latin2]{inputenc}
\usepackage{t1enc}
\usepackage[mathscr]{eucal}
\usepackage{indentfirst}
\usepackage{graphicx}
\usepackage{graphics}
\usepackage{epic}
\numberwithin{equation}{section}
\usepackage[margin=2.9cm, marginparwidth=2.4cm]{geometry}
\usepackage{epstopdf} 
\usepackage{enumerate}
\usepackage{mathrsfs} \usepackage{microtype}
\usepackage{cases}

\usepackage[colorlinks=true, pdfstartview=FitV, linkcolor=blue, citecolor=blue, urlcolor=blue,pagebackref=false]{hyperref}

\theoremstyle{plain}
\newtheorem{Th}{Theorem}[section]
\newtheorem{Lemma}[Th]{Lemma}
\newtheorem{Cor}[Th]{Corollary}
\newtheorem{Prop}[Th]{Proposition}

 \theoremstyle{definition}
\newtheorem{Def}[Th]{Definition}
\newtheorem*{Def*}{Definition}

\newtheorem{Rem}[Th]{Remark}
\newtheorem*{Rem*}{Remark}
\newtheorem{?}[Th]{Problem}

\newcommand{\E}{\mathbb{E} }
\renewcommand{\P}{\mathbb{P}}

\newcommand{\R}{\mathbb{R}}
\newcommand{\N}{\mathbb{N}}

\renewcommand{\d}{\mathrm{d}}
\newcommand{\la}{\langle}\newcommand{\ra}{\rangle}
\newcommand{\eps}{\epsilon}

\newcommand{\bF}{\overline{F}}
\renewcommand{\div}{\nabla \cdot}
\renewcommand{\H}{\mathsf{H}}

\newcommand{\tx}{\tilde{x}}

\newcommand{\bY}{\overline{Y}}
\newcommand{\ii}{\mathbf{i}}
\newcommand{\jj}{\mathbf{j}}
\renewcommand{\S}{\mathbf{S}}
\newcommand{\sym}{\mathsf{sym}}

\newcommand{\SM}{\S^K_{+,  M}}
\newcommand{\cK}{\mathcal{K}}
\newcommand{\bcK}{\overline{\cK}}
\newcommand{\cL}{\mathcal{L}}
\newcommand{\cD}{\mathcal{D}}
\newcommand{\tphi}{\widetilde{\phi}}
\newcommand{\bFN}{\bF_N}
\newcommand{\tN}{t_N}
\newcommand{\hN}{h_N}
\newcommand{\cS}{\mathcal{S}}
\newcommand{\tr}{\mathsf{tr}}
\newcommand{\nn}{\mathbf{n}}
\newcommand{\cP}{\mathcal{P}}
\newcommand{\cA}{\mathcal{A}}

\newcommand{\cC}{\mathcal{C}}
\newcommand{\itr}{\mathsf{int}\,}
\newcommand{\cl}{\mathsf{cl}\,}
\newcommand{\bd}{\mathsf{bd}\,}
\newcommand{\conv}{\mathsf{conv}\,}
\newcommand{\rank}{\mathsf{rank}}
\newcommand{\zero}{\mathbf{0}}
\newcommand{\diag}{\mathsf{diag}}

\usepackage{ulem}
\usepackage{cancel}

\begin{document}

\author[Hong-Bin Chen]{Hong-Bin Chen}
\address[Hong-Bin Chen]{Courant Institute of Mathematical Sciences, New York University, New York, New York, USA}
\email{hbchen@cims.nyu.edu}

\author[Jiaming Xia]{Jiaming Xia}
\address[Jiaming Xia]{Department of Mathematics, University of Pennsylvania, Philadelphia, Pennsylvania, USA}
\email{xiajiam@sas.upenn.edu}

\keywords{inference problem, Hamilton-Jacobi equation, tensor}
\subjclass[2010]{82B44, 82D30}

\title{Hamilton--Jacobi equations for inference of matrix tensor products}

\begin{abstract}
We study the high-dimensional limit of the free energy associated with the inference problem of finite-rank matrix tensor products. In general, we bound the limit from above by the unique solution to a certain Hamilton--Jacobi equation. Under additional assumptions on the nonlinearity in the equation which is determined explicitly by the model, we identify the limit with the solution. Two notions of solutions, weak solutions and viscosity solutions, are considered, each of which has its own advantages and requires different treatments. For concreteness, we apply our results to a model with i.i.d.\ entries and symmetric interactions. In particular, for the first order and even order tensor products, we identify the limit and obtain estimates on convergence rates; for other odd orders, upper bounds are obtained.
\end{abstract}
\maketitle

\section{Introduction}
Tensor factorizations or tensor decompositions play important roles in numerous applications.  In this work, we study the inference problem of estimating tensor products of matrices. Let us first describe the model we are concerned with. Fix $K\in\N$ and let $P^X_N$ be the law of $X\in \R^{N\times K}$, where $N\in\N$ will be sent to $\infty$. 
For a fixed $L\in\N$, we observe
\begin{align}\label{eq:observ}
    Y= \sqrt{\frac{2t}{N^{p-1}}}X^{\otimes p}A +W \ \  \in \R^{N^p\times L}.
\end{align}
where $t\geq0$ is interpreted as the signal-to-noise ratio;  $\otimes$ is the Kronecker product (hence $X^{\otimes p}\in\R^{N^p\times K^p}$); $A\in\R^{K^p\times L}$ is a deterministic matrix; and $W\in \R^{N^p\times L}$ consists of independent standard Gaussian entries.

\smallskip

The inference task is to recover the information of $X$ based on the observation of $Y$. Hence, we investigate the law of $X$ conditioned on observing $Y$. Bayes' rule gives that, for any bounded measurable $g:\R^{N\times K}\to\R$, we have
\begin{align*}
    \E\big[g(X)\big|Y\big] = \frac{\int_{\R^{N\times K}}g(x)e^{H^\circ_N(t,x)}P^X_N(\d x)}{\int_{\R^{N\times K}}e^{H^\circ_N(t,x)}P^X_N(\d x)}.
\end{align*}
Here the Hamiltonian associated with this model is given by
\begin{align}\label{eq:H^circ}
    H^\circ_N(t,x) = \sqrt{\frac{2t}{N^{p-1}}}(x^{\otimes p}A) \cdot Y - \frac{t}{N^{p-1}}|x^{\otimes p}A|^2.
\end{align}
Throughout this paper, the dot product between two tensors, matrices or vectors of the same size is the entry-wise inner product. We denote by $|\cdot|$ the associated norm.
The goal is to understand the high-dimensional limit as $N\to\infty$ of the free energy
\begin{align*}
    \E F^\circ _N(t)= \frac{1}{N}\E\log \int_{\R^{N\times K}} e^{H^\circ_N(t,x)}P^X_N(\d x).
\end{align*}

\bigskip

We briefly discuss the generality of the model \eqref{eq:observ} and its relation to other models involving the inference of matrix products. Among the ones widely studied are the models concerning the second order products. The inference problem of nonsymmetric matrices (or the spiked Wishart model) is given by $Y = \sqrt{\frac{2t}{N}}X_1X_2^\intercal+W$. Works investigating this model include \cite{miolane2017fundamental,barbier2017layered,barbier2019adaptive,kadmon2018statistical,luneau2020high,chen2020hamilton}. When $X_1=X_2$, this becomes the inference problem of symmetric matrices (or the spiked Wigner model), which is studied in \cite{lelarge2019fundamental,dia2016mutual,mourrat2018hamilton,mourrat2019hamilton}. A generalization of these spiked matrix models can be seen in the study of community detection problems and the stochastic block models. In certain settings, the community detection problem is asymptotically equivalent to $Y= \sqrt{\frac{2t}{N}}XBX^\intercal +W$ where $B$ is deterministic and models the community interactions (see \cite{reeves2019geometryIEEE,reeves2019geometry}). More generally, the community detection with several correlated networks is asymptotically equivalent to the multiview spiked matrix model $Y_l= \sqrt{\frac{2t}{N}}XB_lX^\intercal +W_l$ for $l=1,2,\dots,L$ where each $B_l$ reflects one network (see \cite{mayya2019mutualIEEE,mayya2019mutual}). All of these second order models can be represented in the form of $Y=\sqrt{\frac{2t}{N}}X^{\otimes 2}\sqrt{S}+W$ where $S$ is a positive semidefinite matrix. This model is studied in \cite{reeves2020information}, and its equivalence to the models above is discussed in more details therein. Hence, the models so far mentioned can be seen as special cases of \eqref{eq:observ} for $p=2$. In Appendix~\ref{section:apdx_nonsym}, we will demonstrate the representation of the nonsymmetric matrix inference problem into the form of \eqref{eq:observ}. Higher order cases ($p\geq 2$) include $Y=\sqrt{\frac{2t}{N^{p-1}}}X^{\otimes p}+W$ with vector $X \in \R^N$ in \cite{barbier2019adaptive,mourrat2018hamilton}, and $Y=\sqrt{\frac{2t}{N^{p-1}}}\sum_{k=1}^rX_k^{\otimes p}+W$ with each vector $X_k \in \R^N$ in \cite{lesieur2017statistical}. The model \eqref{eq:special} studied in \cite{luneau2019mutual} and considered in Section~\ref{section:special} as a special case also belongs to this class. Again, they can be viewed as special cases of \eqref{eq:observ}.

\smallskip

Recently, the powerful method of adaptive interpolations was introduced in \cite{barbier2019adaptive}. This technique and its improvements have been employed in works including \cite{barbier2019optimal,luneau2020high,reeves2020information}. 
In this work, we follow the approach via Hamilton--Jacobi equations set forth in \cite{mourrat2018hamilton,mourrat2019hamilton,mourrat2019parisi,mourrat2020extending,mourrat2020nonconvex,mourrat2020free}. Let $F_N(t,h)$ be the free energy corresponding to an enriched version of the Hamiltonian \eqref{eq:H^circ}. Here $h$ is an additional variable and the original free energy satisfies $F^\circ_N(t)=F_N(t,0)$. We seek to compare the limit of $\E F_N(t,h)$ as $N\to \infty$ with the solution of the following Hamilton--Jacobi equation
\begin{align*}
    \big(\partial_t f -\H(\nabla f)\big)(t,h)=0.
\end{align*}
Here the nonlinearity $\H$ is given by a simple formula \eqref{eq:H} in terms of the interaction matrix $A$ in \eqref{eq:observ}. To make sense of solutions of this equation and the convergence, two notions have been explored. The notion of viscosity solutions of Hamilton--Jacobi equations was initially adopted to study convergence of free energies in \cite{mourrat2018hamilton} and later the notion of weak solutions was taken in \cite{mourrat2019hamilton}. Viscosity solutions are in general heavier to handle. Bounds from two sides require different treatments, and often one side is much easier than the other and requires weaker assumptions. The convergence happens in the local $L^\infty_tL^\infty_h$ topology while it takes considerable effort to obtain convergence rates. On the other hand, weak solutions are simpler and it is easier to obtain estimates on convergence rates, although the convergence takes place in local $L^\infty_tL^1_h$. It can be upgraded to estimates in $L^\infty_tL^\infty_h$ by giving up some powers (see Remark~\ref{Rem:infinity_norm}). A more detailed comparison of these two notions of solutions can be found in \cite[Section 2]{mourrat2019hamilton}.

\smallskip

We utilize both notions in this work. For any interaction matrix $A$ (equivalently, for any $\H$ of the form \eqref{eq:H}), we obtain an upper bound on the limit of the free energy in Theorem~\ref{thm:lower} via viscosity solutions. This theorem also gives the corresponding lower bound under an additional assumption that $\H$ is convex. Employing weak solutions as in Theorem~\ref{thm:general_cvg_bF}, we obtain convergence and estimates on convergence rates under an assumption on $\H$ which is weaker than convexity. 

\smallskip

We emphasize that, different from the usual approach in statistical mechanics, the existence of a variational formula for the limit of free energies is not \textit{a priori} needed in our approach. Instead, the existence of solutions to the Hamilton--Jacobi equation is sufficient. In the weak solution approach, we prove the existence in a straightforward manner by verifying that the free energies form a Cauchy sequence. For viscosity solutions, there are classical tools to ensure existence. Here, we prove that the Hopf formula is a viscosity solution as a useful fact (see Remark~\ref{Rem:var_formula}), and simply use this to furnish the existence for convenience.

\smallskip

The rest of the paper is organized as follows. We describe the setting and state main results in Section~\ref{section:setting_main}. We apply these results to a special case where $X$ has i.i.d.\ entries and the interaction is symmetric in Section~\ref{section:special}. In Section~\ref{section:approx_HJ}, we show that the free energy satisfies an approximate Hamilton--Jacobi equation and collect some basic results of the derivatives of the free energy. Section~\ref{section:weak_sol} gives the precise definition of weak solutions and the uniqueness of solutions. In Section~\ref{section:cvg_weak}, we show the convergence of the free energy to a weak solution, and finish the proof of Theorem~\ref{thm:general_cvg_bF}. The definition of viscosity solutions and the corresponding well-posedness results are in Section~\ref{section:vis_sol}. The ensuing Section~\ref{section:cvg_vis_sol} studies the convergence of the free energy to the viscosity solution and proves Theorem~\ref{thm:lower}. A special version of the Fenchel--Moreau biconjugation theorem on the set of positive semidefinite matrices is needed to analyze the Hopf formula. It is stated and proved in Appendix~\ref{section:biconjugation}.

\subsection*{Acknowledgement} We warmly thank Jean--Christophe Mourrat for many helpful comments and discussions.

\section{Setting and Main results}\label{section:setting_main}

\subsection{Setting}
 We assume that the random matrix $X\in\R^{N\times K}$ in \eqref{eq:observ} satisfies
\begin{align}\label{eq:support}
    |X|\leq \sqrt{NK}.
\end{align}
For convenience, we use the shorthand notation
\begin{align}\label{eq:tx}
    \tx = x^{\otimes p}A,\quad\forall  x\in\R^{N\times K}.
\end{align}
We enrich the Hamiltonian \eqref{eq:H^circ} by introducing
\begin{align}\label{eq:H_N_enriched}
\begin{split}
    H_N(t,h,x) &= \sqrt{\frac{2t}{N^{p-1}}}\tx \cdot Y - \frac{t}{N^{p-1}}|\tx|^2 \\
    &\quad+\sqrt{2h}\cdot (x^\intercal \bY) - h\cdot (x^\intercal x).
\end{split}
\end{align}
Here $\overline{Y}=X\sqrt{2h}+Z$, where $h\in\S^K_+$, the set of $K\times K$ (symmetric) positive semi-definite matrices, and entries of $Z\in\R^{N\times K}$ are independent standard Gaussian variables. This Hamiltonian $H_N$ is associated with the law of $X$ conditioned on observing both $Y$ and $\overline{Y}$.
The corresponding free energy is given by 
\begin{align}\label{eq:def_F_N}
    F_N(t,h) =\frac{1}{N} \log \int_{\R^{N\times K}} e^{H_N(t,h,x)}P^X_N(\d x).
\end{align}
Let $\bF_N(t,h) = \E F_N(t,h)$ be its expectation.

\smallskip

Set $\R_+=[0,\infty)$. We consider the Hamilton--Jacobi equation
\begin{align}\label{eq:HJ_eqn}
    \partial_t f - \H(\nabla f) =0, \quad \text{in } \R_+ \times \S^K_+
\end{align}
where $\H:\S^K_+\to \R$ is given by
\begin{align}\label{eq:H}
    \H(q) = \big(AA^\intercal\big)\cdot q^{\otimes p},\quad \forall q\in\S^K_+.
\end{align}

\subsection{Main results}
To state the results, we need more notation. Let us introduce
\begin{align}\label{eq:def_SM}
    \SM = \big\{ h \in \S^K_+:\ |h|\leq M\big\}.
\end{align}
We also denote the set of $K\times K$ symmetric matrices by $\S^K$, and the set of $K\times K$ symmetric positive definite matrices by $\S^K_{++}$.
For $N\in \N$ and $M>0$, we define
\begin{align}\label{eq:def_K}
    \cK_{M,N} = \bigg(\E \sup_{(t,h)\in [0,M]\times \SM}\big| F_N - \bF_N|^2\bigg)^\frac{1}{2},
\end{align}
and for any function $\psi:\S^K_+\to \R$,
\begin{align}\label{eq:def_L}
    \cL_{\psi,M,N}=\sup_{h\in \SM}\big|\bF_N(0,h ) -\psi(h)\big|.
\end{align}
The quantity $\cK_{M,N}$ measures the concentration of $F_N$. Many tools are available to estimate this. In view of \eqref{eq:H_N_enriched} and \eqref{eq:def_F_N}, we can recast $\bF_N(0,h)$ as the free energy corresponding to a decoupled system (inference of $X$ based on the observation of $\overline{Y}$ with $\overline{Y}$ in \eqref{eq:H_N_enriched}). Hence, $\cL_{\psi,M,N}$ is also a relatively simple object to analyze.

\smallskip

Throughout, the gradient $\nabla$ is taken in the space variable $h\in \S^K_+$ (sometimes written as $x\in\S^K_+$). To avoid confusion when multiple $\nabla$ are present, we specifically denote the differential of $\H$ by $\cD\H$. We identify $\S^K$ with $\R^{K(K+1)/2}$ in an isometric way (see \eqref{eq:orthonormal_basis}) and endow it with the Lebesgue measure. Let $\cA$ be the set of real-valued nondecreasing, Lipschitz and convex functions on $\S^K_+$. Here a function $u:\S^K_+\rightarrow \R$ is said to be nondecreasing provided 
\begin{align}\label{eq:nondecreasing}
    u(a)\geq u(b),\quad\text{ if } a-b\in\S^K_+. 
\end{align}
We define
\begin{align}\label{eq:A_H}
    \cA_\H=\Big\{\phi\in\cA:\  \nabla\cdot \big(\cD\H(\nabla\phi)\big)\geq 0 \Big\},
\end{align}
where the inequality is understood in the sense of distribution, namely $\int \cD\H(\nabla\phi)\cdot \nabla \eta \leq 0$, for all nonnegative smooth function $\eta$ compactly supported on $\S^K_{++}$.

\smallskip

Before stating the theorems, we comment that the assumptions imposed in them are threefold. The first part is on the concentration, namely, the quantity $\cK_{M,N}$. The second part is on $\bF_N(0,\cdot)$ or $\cL_{\psi,M,N}$, which is about the convergence of the free energy in the aforementioned decoupled system. The third part is on $\H$ (equivalently on $A$ due to \eqref{eq:H}) or, further, on $\cA_\H$. 

\begin{Th}\label{thm:general_cvg_bF}
Let $p\in\N$. Suppose
\begin{itemize}
    \item $\sup_{M\geq 1, N\in\N}(\cK_{M,N}/M^\beta)<\infty$ for some $\beta>0$, and $\lim_{N\to\infty} \cK_{M,N}=0$ for each $M\geq 1$;
    
    \item there is a function $\psi:\S^K_+\to \R$ such that $\lim_{N\to\infty}\cL_{\psi,M,N}=0$ for each $M\geq 1$;

\item $\cA_\H$ is convex and $\bF_N(t,\cdot)\in\cA_\H$ for all $t\geq 0$ and $N\in\N$.
\end{itemize}

Then there is a unique weak solution $f$ to \eqref{eq:HJ_eqn} with $f(0,\cdot)=\psi$, and there is a constant $C>0$ such that the following holds for all $M\geq 1$ and all $N\in\N$:
\begin{align}\label{eq:cvg_rate}
    \sup_{t\in[0,M]}\int_{\S^K_{+,M}}\big|\bF_N(t,h)-f(t,h)\big|\d h \leq CM^\alpha\Big(\cL_{\psi,CM,N}+N^{-\frac{1}{14}}+\big(\cK_{CM,N}/M^\beta\big)^\frac{2}{7}\Big),
\end{align}
where $\alpha = \frac{K(K+1)}{2}+\frac{\beta\vee1}{2}+1$.

\end{Th}

\begin{Th}\label{thm:lower}
Let $p\in\N$. Suppose that there is $\psi:\S^K_+\to\R$ such that $\bF_N(0,\cdot)$ converges to $\psi$ pointwise, and that for each $M>0$ we have
\begin{align}\label{eq:cK_decay_beta}
    \lim_{N\rightarrow \infty}\cK_{M,N}=0.
\end{align}
Then, for any $\H$ of the form \eqref{eq:H}, there is a unique Lipschitz viscosity solution $f$ to \eqref{eq:HJ_eqn} with $f(0,\cdot)=\psi$, and 
\begin{align*}
    \limsup_{N\to\infty}\bF_N(t,h) \leq f(t,h),\quad \forall (t,h)\in \R_+\times \S^K_+.
\end{align*}
If, in addition, $\H$ is convex, then a corresponding lower bound holds and thus
\begin{align*}
    \lim_{N\to\infty}\bF_N(t,h) = f(t,h),\quad \forall (t,h)\in \R_+\times \S^K_+.
\end{align*}

\end{Th}

The proofs of Theorem~\ref{thm:general_cvg_bF} and Theorem~\ref{thm:lower} are in Section~\ref{section:cvg_weak} and Section~\ref{section:vis_sol}, respectively.

\begin{Rem}[Conditions on $\cA_\H$]
When $\phi$ is smooth, we can compute that $\nabla \cdot \big(\cD \H(\nabla \phi)\big)= \cD^2\H(\nabla \phi)\cdot \nabla^2\phi$ where $\cD^2\H$ is the Hessian of $\H$. Lemma~\ref{lemma:H_monotone_special} will show that if $\H$ is convex, then the conditions on $\cA_\H$ in Theorem~\ref{thm:general_cvg_bF}, namely, the convexity of $\cA_\H$ and $\bF_N(t,\cdot)\in \cA_\H$, are satisfied. 

Note that when $p\leq 2$, $\cD^2\H$ is constant and in this case $\cA_\H$ is always convex. Hence, the only condition to check is that $\bF_N(t,\cdot)\in\cA_\H$. In Appendix~\ref{section:apdx_nonsym}, we demonstrate a special model of \eqref{eq:observ} with $p=2$ where this condition is satisfied but $\H$ is not convex. This model is equivalent to the nonsymmetric matrix inference problem considered in \cite{miolane2017fundamental,barbier2017layered,barbier2019adaptive,kadmon2018statistical,luneau2020high,chen2020hamilton}.

It seems that the conditions on $\cA_\H$ are not satisfied by the model \eqref{eq:special} for odd $p\geq 1$. The explicit expression of $\cD^2\H$ in this model is computed in \eqref{eq:Hessian_H_special}. We believe that this issue is closely related to a similar difficulty in the adaptive interpolation approach to the same model with odd $p$, which is discussed in \cite[Section 7]{luneau2019mutual}.
\end{Rem}

\begin{Rem}[Local uniform convergence]\label{Rem:infinity_norm}
The local $L^\infty_tL^1_x$ convergence in Theorem~\ref{thm:general_cvg_bF} can be upgraded to local $L^\infty_tL^\infty_x$. Let $\xi$ be a smooth function supported on $-\S^K_{+,1}$, and satisfy $0\leq \xi\leq 1$ and $\int \xi>0$. For $\eps\in(0,1)$, let $\xi_\eps(x) = \eps^{-K(K+1)/2}\xi(\eps^{-1}x)$. Then, for every Lipschitz $g:\S^K_+\to\R$, we have
\begin{align*}
    \|g\|_{L^\infty(\S^K_{+,M})}\leq \|g*\xi_\eps\|_{L^\infty(\S^K_{+,M})}+ \|g-g*\xi_\eps\|_{L^\infty(\S^K_{+,M})}\\
    \leq C\eps^{-K(K+1)/2}\|g\|_{L^1(\S^K_{+,M+1})}+C\eps\|g\|_{\mathrm{Lip}}.
\end{align*}
By \eqref{eq:1st_der_bF}, we know $\bF_N(t,\cdot)$ is Lipschitz uniformly in $N$ and $t$, and thus $f(t,\cdot)$ is also Lipschitz. Replace $g$ in the above by $\bF_N(t,\cdot)-f(t,\cdot)$, apply Theorem~\ref{thm:general_cvg_bF} and optimize the above display over $\eps$ to see convergence in local $L^\infty_tL^\infty_x$.
\end{Rem}

\begin{Rem}[Variational formulae]\label{Rem:var_formula}
Under the assumptions on $\psi$ in the two theorems, we can show that $\psi$ is Lipschitz, convex and nondecreasing in the sense that $\nabla\psi\in\S^K_+$. By the pointwise convergence $\bF_N(0,\cdot)\to\psi$ and \eqref{eq:1st_der_bF}, \eqref{eq:1st_der_bF_lower}, \eqref{eq:2nd_der_bF}, and the pointwise convergence $\bF_N(0,\cdot)\to\psi$, we can see that $\psi$ is Lipschitz in the two theorems above. Proposition~\ref{Prop:hopf_viscosity} will show that $f$ in Theorem~\ref{thm:lower} can be represented by the following variational formula
\begin{align}\label{eq:f_Hopf_0}
    f(t,x)=\sup_{z\in \S^K_+}\inf_{y\in \S^K_+}\big\{ z\cdot (x-y)+\psi(y)+t\H(z)\big\},\quad \forall (t,x)\in\R_+\times \S^K_+.
\end{align}

When $\H$ is convex, comparing Theorem~\ref{thm:general_cvg_bF} with Theorem~\ref{thm:lower} in view of Remark~\ref{Rem:infinity_norm}, we can see that the unique weak solution $f$ coincides with the viscosity solution pointwise, and thus also admits the representation \eqref{eq:f_Hopf_0}. For general $\H$, we believe weak solutions are still of the form \eqref{eq:f_Hopf_0}. The relatively difficult part is to verify that \eqref{eq:f_Hopf_0} satisfies \eqref{item:ws_2} of Definition~\ref{Def:weak_sol}. 

\end{Rem}

\begin{Rem}[Possibility for weaker assumptions on $\H$]
Let us point out key inequalities, where the assumptions on $\H$ are used. If these inequalities still hold in certain models, then our results should still be valid there.

The conditions on $\cA_\H$ in Theorem~\ref{thm:general_cvg_bF} are used to obtain the inequality \eqref{eq:div_b_eps_geq_0} in the proof of Lemma~\ref{lemma:weak_comp}, which is further used to prove the uniqueness of weak solutions (Proposition~\ref{Prop:uniq_HJ}), and the convergence to the unique weak solution (Proposition~\ref{prop:cvg_assume_existence} and Proposition~\ref{prop:existence}). In fact, uniqueness and convergence are still valid if the right-hand side of \eqref{eq:div_b_eps_geq_0} is replaced by a negative constant depending locally on the temporal and spacial variables. However, the convergence rate can be much worse (logarithmic in $N$), because the absolute value of this constant will appear in the exponential factor of Gronwall's lemma.

The convexity assumption in the second assertion of Theorem~\ref{thm:lower} is only used to apply Jensen's inequality to derive \eqref{eq:supsol_G_N_lower} in the proof of that the limit of $\bF_N$ is a viscosity supersolution. 
\end{Rem}

\subsection{Special case}\label{section:special}

We apply Theorem~\ref{thm:general_cvg_bF} and Theorem~\ref{thm:lower} to an i.i.d.\ case. Let $\cP$ be a probability distribution in $\R^K$ supported on $\{z\in\R^K:|z|\leq \sqrt{K}\}$. For each $N\in\N$, let the row vectors of $X$, namely $X_{1,\cdot}\ ,X_{2,\cdot}\ ,\dots,\ X_{N,\cdot}$, be i.i.d.\ with law $\cP$. Set $L=1$ and consider
$A \in \R^{K^p\times 1}$ given by
\begin{align*}
    A_\jj = 
    \begin{cases}
    1,&\quad \text{if $j_1=j_2=\dots = j_p$,}\\
    0, &\quad \text{otherwise.}
    \end{cases}
\end{align*}
Here, we used the multi-index notation
\begin{gather}\label{eq:multi-index}
    \jj = (j_1,j_2,\dots,j_p)  \in \{1,\cdots, K\}^p.
\end{gather}
Explicitly, \eqref{eq:observ} now becomes
\begin{align}\label{eq:special}
    Y_\ii = \sqrt{\frac{2t}{N^{p-1}}} \sum_{j=1}^K\prod_{n=1}^pX_{i_n, j}+W_\ii,\quad \ii \in \{1,\cdots, N\}^p,
\end{align}
and \eqref{eq:H} becomes
\begin{align}\label{eq:H_special}
    \H(q) = \sum_{j,j'=1}^K\big(q_{j,j'}\big)^p,\quad q \in \S^K_+.
\end{align}

Using \eqref{eq:def_F_N} and the fact that rows of $X$ are i.i.d., we can see $\bF_N(0,\cdot) = \bF_1(0,\cdot)$,
for all $N\in\N$. Setting $\psi = \bF_1(0,\cdot)$, we clearly have $\cL_{\psi,M,N}=0$ for all $M$ and $N$. Estimate on $\cK_{M,N}$ is given in Lemma~\ref{lemma:concentration}. When $p=1$ or $p$ is even, Lemma~\ref{lemma:H_monotone_special} shows that the assumptions on $\cA_\H$ in Theorem~\ref{thm:general_cvg_bF} are satisfied. Applying the main results, we have the following corollary.

\begin{Cor}
In the special case described above, let $f$ be given by \eqref{eq:f_Hopf_0} with $\psi = \bF_1(0,\cdot)$. Then for all $p\in\N$, we have
\begin{align*}
    \limsup_{N\to\infty}\bF_N(t,h)\leq f(t,h),\quad \forall(t,h)\in \R_+\times \S^K_+.
\end{align*}
If $p$ is even or $p=1$, then there is $C>0$ such that, for all $M\geq 1$ and $N\in\N$,
\begin{align*}
    \sup_{t\in[0,M]}\int_{\S^K_{+,M}}\big|\bF_N(t,h)-f(t,h)\big|\d h \leq C M^{\frac{K(K+1)+3}{2}}N^{-\frac{1}{14}}.
\end{align*}
\end{Cor}

This model \eqref{eq:special} has also been investigated in \cite{luneau2019mutual} and similar convergence results for even orders were established. Although the convergence for odd orders remains open, we are able to obtain an upper bound for the limit of free energy.

\section{Approximate Hamilton--Jacobi equations}\label{section:approx_HJ}

The goal of this section is to show that $\bF_N$ satisfies an approximate Hamilton--Jacobi equation, as summarized in Proposition~\ref{Prop:approx_HJ} below. There is a considerable overlap between results in this section and \cite[Section 3]{mourrat2019hamilton}, which follows the approach of \cite{barbier2019overlap}. To simplify our presentation, whenever similar arguments are available in \cite[Section 3]{mourrat2019hamilton}, we shall only demonstrate key steps and refer to \cite[Section 3]{mourrat2019hamilton} for more detailed computations.

\begin{Prop}[Approximate Hamilton--Jacobi equations]\label{Prop:approx_HJ}
There exists $C>0$ such that for every $N\geq 1$ and uniformly over $\R_+\times \S^K_+$,
\begin{equation*}
\big|\partial_t\bF_N-\H(\nabla\bF_N)\big|^2\leq C\kappa(h)N^{-\frac{1}{4}}\big(\Delta \bF_N+|h^{-1}|\big)^{\frac{1}{4}}+C\E\big|\nabla F_N-\nabla \bF_N\big|^2.
\end{equation*}
Here $\kappa$ is the condition number of $h\in \S^K_+$ given by
\begin{equation}\label{eq:def_kappa(h)}
\kappa(h):=
\left\{
\begin{array}{ll}
|h||h^{-1}|,    &\quad\text{if }h\in \S^K_{++},  \\
+\infty     &\quad\text{otherwise}.
\end{array}
\right.
\end{equation}
\end{Prop}

\subsection{Proof of Proposition~\ref{Prop:approx_HJ}}

We start by proving the following identity
\begin{align}\label{eq:identity_bF_N}
    \partial_t\bF_N - \H\big(\nabla \bF_N\big)=\frac{1}{N^p}\bigg(\E\big\la\H(x^\intercal x') \big\ra - \H\big(\E\la x^\intercal x' \ra\big)\bigg).
\end{align}
\begin{proof}[Proof of \eqref{eq:identity_bF_N}]
Let us first compute $\partial_t\bF_N$ and $\nabla\bFN$.
Indeed, from \eqref{eq:def_F_N}, we can compute
\begin{equation}\label{eq:dtFN}
    \partial_t F_N(t,h)=\frac{1}{N}\bigg\la \frac{2}{N^{p-1}}\tilde x\cdot \tilde X+\sqrt{\frac{1}{2N^{p-1}t}}\tilde x \cdot W-\frac{|\tilde x|^2}{N^{p-1}}\bigg\ra,
\end{equation}
and, for $a\in \S^K$,
\begin{equation}\label{eq:nablaFN}
    a\cdot \nabla F_N(t,h)=
    \frac{1}{N}\Big\la 2a\cdot (x^\intercal X)+\sqrt{2}D_{\sqrt{h}}(a)\cdot (x^\intercal Z)-a\cdot (x^\intercal x)\Big\ra.
\end{equation}
Here $D_{\sqrt{h}}$ is the differential of the square-root function at $h\in \S^K_{++}$. More precisely, for $h\in \S^K_{++}$ and $a\in \S_K$, we have
\begin{equation*}
    D_{\sqrt{h}}(a)=\lim_{\eps\rightarrow 0}\Big(\sqrt{h+\eps a}-\sqrt{h}\Big).
\end{equation*}

\smallskip

Using the Gaussian integration by parts (c.f.\ \cite[Lemma 3.3]{mourrat2019hamilton}) and the Nishimori identity (c.f.\ \cite[Section 3.1]{mourrat2019hamilton}), we can get from \eqref{eq:dtFN} that
\begin{align}\label{eq:dtbFN}
    \partial_t\bF_N = \frac{1}{N^p}\E \la \tx \cdot \tx'\ra .
\end{align}
Here $x'$ is an independent copy (or replica) of $x$ with respect to the Gibbs measure $\la \cdot \ra$.

\smallskip

To compute $\nabla\bF_N$, we refer to the derivation of \cite[(3.17)]{mourrat2019hamilton}. The object $\bar x$ therein is $X$ in our notation, and our $F_N(t,h)$ corresponds to $F_N(t,2h)$ there. Hence \cite[(3.17)]{mourrat2019hamilton} is equivalent to $\nabla \bF_N = \frac{1}{N}\E\la x^\intercal X\ra$. A further application of the Nishimori identity yields
\begin{align}\label{eq:grad_bFN}
    \nabla\bF_N = \frac{1}{N}\E\la x^\intercal x'\ra.
\end{align}
By \eqref{eq:tx} and \eqref{eq:H}, we have $\tx\cdot \tx' = \H(x^\intercal x')$. This along with \eqref{eq:dtbFN}, \eqref{eq:grad_bFN} and \eqref{eq:H} implies \eqref{eq:identity_bF_N}.
\end{proof}

\medskip

Now, to prove Proposition~\ref{Prop:approx_HJ}, we only need to estimate the right hand side of \eqref{eq:identity_bF_N}. Using \eqref{eq:H} and \eqref{eq:support}, we get
\begin{align*}
    \Big|\E\big\la\H(x^\intercal x')\big\ra-\H\big(\E\la x^\intercal x'\ra\big)\Big|&\leq C\E\Big\la\Big|\big(x^\intercal x'\big)^{\otimes p}-\big(\E\la x^\intercal x'\ra\big)^{\otimes p}\Big|\Big\ra\\
    &\leq CN^{p-1}\E\big\la\big|x^\intercal x'-\E\la x^\intercal x'\ra\big|\ra,
\end{align*}
Jensen's inequality gives
\begin{align*}
    \Big|\E\big\la\H(x^\intercal x')\big\ra-\H\big(\E\la x^\intercal x'\ra\big)\Big|^2&\leq
    CN^{2p-2}\E\Big\la \big|x^\intercal x'-\E\la x^\intercal x'\ra\big|^2\Big\ra.
\end{align*}
We need the following estimate
\begin{equation*}
    \frac{1}{N^2}\E\Big\la \big|x^\intercal x'-\E\la x^\intercal x'\ra\big|^2\Big\ra\leq C\kappa(h)N^{-\frac{1}{4}}\big(\Delta \bF_N+|h^{-1}|\big)^{\frac{1}{4}}+C\E\big|\nabla F_N-\nabla \bF_N\big|^2.
\end{equation*}
This is exactly \cite[(3.18)]{mourrat2019hamilton}, and we shall omit the derivation here. The above two displays and \eqref{eq:identity_bF_N} gives the desired result.

\subsection{Estimates of derivatives}
We finish this section by collecting useful results in Lemma~\ref{lemma:basic_est} and \eqref{lemma:psd}.
Recall $A\in\R^{K^p\times L}$ and $W\in \R^{N^p\times L}$. We define
\begin{align}\label{eq:||w||}
    \|WA^\intercal\|=\sup_{y_1,\, y_2,\, \dots,\, y_p\in \mathbb{S}^{NK-1}}\Big\{ (WA^\intercal )\cdot (y_1\otimes y_2\otimes \cdots \otimes y_p) \Big\}
\end{align}
where $\mathbb{S}^{NK-1}$ denotes the unit sphere in $\R^{NK}$.
\begin{Lemma}\label{lemma:basic_est}
There exists a constant $C>0$ such that the following estimates hold uniformly over $\R_+\times \S^K_+$ for every $N\in\N$:
\begin{gather}
    |\partial_t\bF_N|+|\nabla\bF_N|\leq C,\label{eq:1st_der_bF}\\
    |\partial_t F_N|\leq C\bigg(1+\frac{\|WA^\intercal\|}{\sqrt{Nt}}\bigg),\quad \text{and}\quad |\nabla F_N|\leq C\bigg(1+\frac{|Z||h^{-1}|^{\frac{1}{2}}}{\sqrt{N}}\bigg).\label{eq:1st_der_F}
\end{gather}
Everywhere in $\R_+\times \S_+^K$, we have
\begin{gather}
    \partial_t\bF_N\geq 0,\qquad \nabla \bF_N\in \S^K_+,\label{eq:1st_der_bF_lower}\\
    \partial_t^2\bF_N \geq 0. \label{eq:2nd_dt_bF}
\end{gather}
Moreover, for every $a\in \S^K$, we have
\begin{gather}
    a\cdot \nabla(a\cdot \nabla\bF_N)\geq 0,\label{eq:2nd_der_bF}\\ 
    \quad a\cdot \nabla(a\cdot \nabla F_N)\geq -\frac{C|a|^2|Z||h^{-1}|^{\frac{3}{2}}}{\sqrt{N}}.\label{eq:2nd_der_F}
\end{gather}
\end{Lemma}

\begin{proof}[Proof of \eqref{eq:1st_der_bF}]
It follows easily from \eqref{eq:support}, \eqref{eq:dtbFN} and \eqref{eq:grad_bFN}.
\end{proof}

\begin{proof}[Proof of \eqref{eq:1st_der_F}]
In view of \eqref{eq:||w||}, we have
\begin{align*}
    \Big|\big(x^{\otimes p}A\big)\cdot W \Big|=\Big|\big(WA^\intercal\big)\cdot \big(x^{\otimes p}\big)\Big|\leq \|WA^\intercal \||x|^p.
\end{align*}
In addition, it can be seen from \eqref{eq:tx} that $|\tx|\leq C|x|^p$.
Using these, (\ref{eq:dtFN}) and (\ref{eq:support}), we have
\begin{align*}
    \big|\partial_t F_N(t,h)\big|&\leq \bigg\la\frac{2}{N^p}|\tilde x||\tilde X| +\sqrt{\frac{1}{2tN^{p+1}}}\Big|\big(x^{\otimes p}A\big)\cdot W \Big|+\frac{1}{N^p} |\tilde x|^2\bigg\ra\\
    &\leq C+\frac{C\|WA^\intercal\|}{\sqrt{Nt}}+C = C\bigg(1+\frac{\|WA^\intercal\|}{\sqrt{Nt}}\bigg).
\end{align*}
For the second estimate in \eqref{eq:1st_der_F}, we need the following estimate
\begin{equation}\label{eq:Dh(a)}
    |D_{\sqrt{h}}(a)|\leq C|a||h^{-1}|^{\frac{1}{2}}.
\end{equation}
Its proof can be seen from the derivation of \cite[(3.7)]{mourrat2019hamilton}. Insert $a=\frac{\nabla F_N}{|\nabla F_N|}\in \S^K$ into \eqref{eq:nablaFN} and then use \eqref{eq:Dh(a)} to see
\begin{align*}
    |\nabla F_N|&\leq \bigg\la \frac{2}{N}\big|x^\intercal X\big| +\frac{C|h^{-1}|^{\frac{1}{2}}}{N}\big|x^\intercal Z\big|+\frac{1}{N} |x^\intercal x|\bigg\ra\leq C\bigg(1+\frac{|Z||h^{-1}|^{\frac{1}{2}}}{\sqrt{N}}\bigg).
\end{align*}
\end{proof}

\begin{proof}[Proof of \eqref{eq:2nd_dt_bF}]

Recall \eqref{eq:dtbFN}. Using \eqref{eq:H_N_enriched}, we differentiate $\partial_t\bF_N$ one more time in $t$ to see
\begin{align*}
    N^p\partial^2_t\bF_N=\E\bigg\la (\tx\cdot\tx')\bigg(\frac{2}{N^{p-1}}\big(\tx+\tx'-2\tx''\big)\cdot \tilde X-\frac{1}{N^{p-1}}\big(|\tx|^2+|\tx'|^2-2|\tx''|^2\big)\\+\frac{1}{\sqrt{2N^{p-1}t}}\big(\tx+\tx'-2\tx''\big)\cdot W\bigg) \bigg\ra.
\end{align*}
Using the symmetry between replicas, the Nishimori identity and the Gaussian integration by parts, we can compute
\begin{align*}
    N^{2p-1}\partial^2_t\bF_N&= 2\E\big\la (\tx\cdot \tx')\big(\tx\cdot\tx' -2\tx\cdot \tx''+ \tx''\cdot \tx'''\big) \big\ra\\
    &=2\E\sum_{\ii,\jj,k,l}\Big(\la \tx_{\ii,k}\tx_{\jj,l}\ra^2 - 2 \la \tx_{\ii,k}\tx_{\jj,l}\ra\la \tx_{\ii,k}\ra\la\tx_{\jj,l}\ra+\la \tx_{\ii,k}\ra^2\la\tx_{\jj,l}\ra^2\Big)\geq 0.
\end{align*}
This gives \eqref{eq:2nd_dt_bF}.
\end{proof}

\begin{proof}[Proof of \eqref{eq:1st_der_bF_lower}]
By the independence of the replica $x'$ from $x$, we can rewrite \eqref{eq:dtbFN} as $\partial_t\bF_N=\frac{1}{N^p}\E(\la \tx\ra \cdot \la\tx\ra)$ and rewrite \eqref{eq:grad_bFN} as $\nabla \bF_N = \frac{1}{N}\E\la x \ra^\intercal \la x \ra$.
Then, \eqref{eq:1st_der_bF_lower} clearly follows.
\end{proof}

\begin{proof}[Proof of \eqref{eq:2nd_der_bF}]
For $a \in \S^K$, we can compute
\begin{align*}
    Na\cdot \nabla(a\cdot \nabla\bF_N) = \E\big\la\big(a\cdot x^\intercal  x'\big)^2\big\ra-2\E\big\la\big(a\cdot x^\intercal  x'\big)\big(a\cdot x^\intercal x''\big)\big\ra+\E\big\la a\cdot x^\intercal x'\big\ra^2.
\end{align*}
The details of this computation can be seen from the derivation of  \cite[(3.27)]{mourrat2019hamilton}. 
Expand the right hand side of the above display to get
\begin{align*}
    \E\sum_{i,j,k,m,n,l}a_{ij}a_{mn}\Big\la x_{ki}x'_{kj}x_{lm}x'_{ln}-2x_{ki}x'_{kj}x_{lm}x''_{ln}+x_{ki}x'_{kj}x''_{lm}x'''_{ln}\Big\ra
\end{align*}
where $x'$, $x''$, $x'''$ are replicas of $x$ with respect to the measure $\la \cdot \ra$. Then, \eqref{eq:2nd_der_bF} follows if we can show the above is nonnegative. Use the independence and write $\hat x = x - \la x \ra$ to see that the above display is equal to
\begin{align*}
    &\E\sum_{i,j,k,m,n,l}a_{ij}a_{mn}\Big(\la x_{ki}x_{lm}\ra\la x_{kj}x_{ln}\ra-2\la x_{ki}x_{lm}\ra\la x_{kj}\ra\la x_{ln}\ra+\la x_{ki}\ra\la x_{lm}\ra\la x_{kj}\ra\la x_{ln}\ra\Big)\\
    =\ &\E\sum_{i,j,k,m,n,l}a_{ij}a_{mn}\Big(\la x_{ki}x_{lm}\ra\la \hat x_{kj}\hat x_{ln}\ra-\la\hat x_{ki}\hat x_{lm}\ra \la x_{kj}\ra\la x_{ln}\ra\Big).
\end{align*}
Notice that since $a\in \S^K$, we can replace $i$ and $m$ by $j$ and $n$, respectively, in the second term inside the last pair of parentheses. So the above becomes
\begin{align*}
     &\E\sum_{i,j,k,m,n,l}a_{ij}a_{mn}\Big(\la x_{ki}x_{lm}\ra\la \hat x_{kj}\hat x_{ln}\ra-\la\hat x_{ki}\hat x_{lm}\ra \la x_{ki}\ra\la x_{lm}\ra\Big)\\
    =\ &\E\sum_{i,j,k,m,n,l}a_{ij}a_{mn}\la\hat x_{ki}\hat x_{lm} \ra\la \hat x_{kj}\hat x_{ln}\ra
    =\E\sum_{i,j,k,m,n,l}a_{ij}a_{mn}\big\la (\hat x^\intercal \hat x')_{ij}(\hat x^\intercal \hat x')_{mn}\big\ra\\
    =\ &\E\la (a\cdot \hat x^\intercal \hat x')^2\ra\geq 0.
\end{align*}
\end{proof}

\begin{proof}[Proof of \eqref{eq:2nd_der_F}]
By \eqref{eq:nablaFN}, we can compute
\begin{align}
    a\cdot &\nabla(a\cdot \nabla F_N(t,h))\nonumber\\
    &=\frac{1}{N}\Big(\Big\la\big(H_N'(a,h,x)\big)^2\Big\ra-\big\la H_N'(a,h,x)\big\ra^2\Big)+\frac{1}{N}\Big\la \sqrt{2}D^2_{\sqrt{h}}(a,a)\cdot x^\intercal Z\Big\ra,\label{eq:aaFN}
\end{align}
where
\begin{equation*}
    H_N'(a,h,x)=\sqrt{2}D_{\sqrt{h}}(a)\cdot x^\intercal Z+2a\cdot x^\intercal X-a\cdot x^\intercal x,
\end{equation*}
and, for every $h\in \S^K_{++}$ and $a,b\in \S^K$, 
\begin{equation*}
    D^2_{\sqrt{h}}(a,b)=\lim_{\eps\rightarrow 0}\eps^{-1}\big(D_{\sqrt{h+\eps b}}(a)-D_{\sqrt{h}}(a)\big).
\end{equation*}
Recognizing a variance term in \eqref{eq:aaFN} and using \eqref{eq:support}, we have
\begin{align*}
    a\cdot \nabla(a\cdot \nabla F_N&(t,h))\geq -C\Big|D^2_{\sqrt{h}}(a,a)\Big|\frac{|Z|}{\sqrt{N}}.
\end{align*}
The display \cite[(3.38)]{mourrat2019hamilton} states 
\begin{equation*}
    \big|D^2_{\sqrt{h}}(a,a)\big|\leq C|a|^2|h^{-1}|^{\frac{3}{2}}.
\end{equation*}
Combining this with the previous display, we obtain \eqref{eq:2nd_der_F}. 
\end{proof}

Lastly, we state an elementary lemma characterizing $\S^K_+$.
\begin{Lemma}\label{lemma:psd}
Let $a\in\S^K$, Then, $a\in\S^K_+$ if and only if $a\cdot b \geq0 $ for every $b\in\S^K_+$.
\end{Lemma}

\begin{proof}
If $a \in \S^K_+$, then for any $b \in\S^K_+$ we have $a\cdot b =\tr(\sqrt{a}\sqrt{b}\sqrt{b}\sqrt{a})\geq 0$. For the other direction, by choosing an orthonormal basis, we may assume $a$ is diagonal. Testing by $b\in\S^K_+$, we can show that all diagonal entries in $a$ are nonnegative and thus $a\in\S^K_+$.
\end{proof}

\section{Weak solutions of Hamilton--Jacobi equations}\label{section:weak_sol}
In this section, we study the Hamilton--Jacobi equation \eqref{eq:HJ_eqn} through the perspective of weak solutions. Precise definitions of weak solutions will be stated and uniqueness of solutions is given in Proposition~\ref{Prop:uniq_HJ}. 

\smallskip

We identify $\S^K$ isometrically with $\R^{K(K+1)/2}$ via the orthonormal basis $\{e^{ij}\}_{1\leq i\leq j\leq K}$ given by, for $m,n \in \{1,2,\dots, K\}$,
\begin{align}\label{eq:orthonormal_basis}
    (e^{ij})_{mn}= \bigg(\mathds{1}_{i=j}+\frac{\sqrt{2}}{2}\mathds{1}_{i\neq j}\bigg)\mathds{1}_{\{m,n\}=\{i,j\}}.
\end{align}
Here $\mathds{1}$ stands for the indicator function. Naturally, we endow $\S^K$ with the Lebesgue measure on $\R^{K(K+1)/2}$. 
Recall the definition of $\cA_\H$ in \eqref{eq:A_H}.

\begin{Def}\label{Def:weak_sol}
A function $f:\R_+\times \S^K_+\to \R$ is a weak solution to \eqref{eq:HJ_eqn} if 
\begin{enumerate}
    \item $f$ is Lipschitz and satisfies \eqref{eq:HJ_eqn} almost everywhere;
    \item \label{item:ws_2} $f(t,\cdot)\in \cA_\H$, for all $t\geq 0$.
\end{enumerate}
\end{Def}

\begin{Prop}[Uniqueness]\label{Prop:uniq_HJ}
Under the assumption that $\cA_\H$ is convex, there is at most one weak solution to \eqref{eq:HJ_eqn}.
\end{Prop}

\subsection{Proof of Proposition~\ref{Prop:uniq_HJ}}
The idea of proof is classical and can be seen in \cite{douglis1965solutions,kruzhkov1966generalized,kruzhkov1967generalized}. See also \cite{benton_hamilton-jacobi_1977} and \cite[Section 3.3.3]{evans2010partial}. The following lemma will also be used later. Recall the definitions of $\S^K_{+,M}$ in \eqref{eq:def_SM}.

\begin{Lemma}\label{lemma:weak_comp}
Assume that $\cA_\H$ is convex. For $M>0, T\geq 1, \eta\in(0,1)$, define
\begin{align}
    D_t& = \S^K_{+,R(T-t)}\cap (\eta I+\S^K_{+}),\quad\forall t\in[0,T] \label{eq:D_t_uniqueness}
\end{align}
with $R = \sup\big\{|\cD\H(p)|:p\in\S^K_{+,M}\big\} $.
Let $\phi:\R\to\R_+$ be any smooth function.
Then, the following holds for all choices of $M, T, \eta, \phi$, and for
every pair $f,g\in \cA_\H$ satisfying $\|f\|_{\mathrm{Lip}}, \|g\|_{\mathrm{Lip}}\leq M$:
\begin{align*}
    \frac{\d}{\d t}J(t) \leq \int_{D_t} \Big(\phi'(f-g)|r|\Big)(t,h)\d h ,\quad\forall t\in [0,T]
\end{align*}
where 
\begin{gather*}
    J(t) = \int_{D_t}\phi(f-g)(t,h)\d h,\\
    r = \Big(\partial_t f - \H(\nabla f)\Big) - \Big(\partial_t g - \H(\nabla g)\Big).
\end{gather*}
\end{Lemma}

\medskip

\begin{proof}
Let us set $w = f- g$ and $v = \phi(w)$. We proceed in steps.

\smallskip

Step 1. We study the relations which $w$ and $v$ satisfy. Since $f$ and $g$ are weak solutions, we have
\begin{align*}
    \partial_t w & = \H(\nabla f)-\H(\nabla g) + r = b\cdot \nabla w +r
\end{align*}
where the function $b$ is given by
\begin{align*}
    b = \int_0^1 \cD\H\big(s \nabla f + (1-s)\nabla g\big)\d s.
\end{align*}
Here $\cD\H$ is the gradient of $\H$ while $\nabla$ is taking derivatives in the spacial variable $h$. Then, we also have
\begin{align}\label{eq:v_eqn}
    \partial_t v = b \cdot \nabla v + \phi'(w)r.
\end{align}

\smallskip

Step 2. We introduce a family of mollifiers. Let $\xi:\R^{K(K+1)/2}\to\R_+$ be smooth, be supported on $-\S^K_{+,1}$, and satisfy $\int \xi=1$. For $\eps\in(0,1)$, set
\begin{align*}
    \xi_\eps =  \eps^{-K(K+1)/2}\xi\Big(\frac{\cdot}{\eps}\Big).
\end{align*}
Define $b_\eps$ by the convolution
\begin{align*}
        b_\eps(t,h) = \big(b(t,\cdot)*\xi_\eps\big)(h)=\int b(t,h-h')\xi_\eps(h')\d h'.
\end{align*}
Recall the definition of $\cA_\H$ in \eqref{eq:A_H}. Since $\cA_\H$ is assumed to be convex and $f,g\in\cA_\H$ are weak solutions, by the definition of $b$, we must have $\nabla\cdot b\geq 0$ in the distribution sense. Then, it is easy to see that
\begin{align}\label{eq:div_b_eps_geq_0}
    \nabla \cdot b_\eps \geq 0
\end{align}
holds pointwise everywhere. We finish this step by proving
\begin{align}\label{eq:b_eps_psd}
    b_\eps \in \S^K_+.
\end{align}
This follows from the next lemma, which will also be used later.
\begin{Lemma}\label{lemma:DH_psd}
For $\H$ given in \eqref{eq:H}, its differential $\cD\H\in \S^K_+$ everywhere.
\end{Lemma}
\begin{proof}
For simplicity, we write $S= AA^\intercal \in \S^{K^p}_+$.
Let $a,q\in \S^K$, then we can compute that
\begin{align*}
    a\cdot \cD\H(q) = pS\cdot \sym(a\otimes q^{\otimes p-1}).
\end{align*}
Here $\sym$ denotes the symmetrization of tensors given by
\begin{align*}
    \sym\Big(b_{1}\otimes b_{2}\otimes \dots \otimes b_{p}\Big) =\frac{1}{p!}\sum_{\sigma} b_{\sigma(1)}\otimes b_{\sigma(2)}\otimes \dots \otimes b_{\sigma(p)},
\end{align*}
where the summation is taken over all permutations. Since $S\in \S^{K^p}_+$, to show $a\cdot \cD\H(q)\geq 0$ it suffices to show $a\otimes q^{\otimes p-1} \in \S^{K^p}_+$. We only need to check
\begin{align*}
    u^\intercal \big(a\otimes q^{\otimes p-1}\big)u \geq 0, \quad \forall u \in \R^{K^p}.
\end{align*}
Index $u \in \R^{K^p}$ as $(u_\ii)_\ii$ with $\ii $ in the form of \eqref{eq:multi-index}. 
Writing $\hat \ii = (i_2,i_3,\dots,i_p)$, let us compute
\begin{align*}
    u^\intercal \big(a\otimes q^{\otimes p-1}\big)u  = \sum_{\ii, \jj}u_\ii \big(a\otimes q^{\otimes p-1}\big)_{\ii,\jj}u_\jj = \sum_{\ii,\jj}u_{i_1,\hat \ii}a_{i_1,j_1}(q^{\otimes p-1})_{\hat\ii,\hat \jj}u_{j_1,\hat\jj}\\ = \tr\big(u^\intercal a u q^{\otimes p-1}\big)=\tr\big(\sqrt{a}uq^{\otimes p-1}u^\intercal\sqrt{a}\big)\geq 0.
\end{align*}
Here, we used the fact that $q^{\otimes p-1}$ is positive semi-definite, which can be proved by iterating the above arguments. Therefore, we can conclude that $a\cdot\cD\H\geq 0$ for every $a\in\S^K_+$, which by Lemma~\ref{lemma:psd} implies $\cD\H\in\S^K_+$.
\end{proof}

\bigskip

Step 3. We study $J(t)$ which can be written as $J(t) = \int_{D_t}v(t,\cdot)$. On $\R_+\times \S^K_{+}$, the equation \eqref{eq:v_eqn} can be expressed as
\begin{align}\label{eq:partial_t_v}
    \partial_t v = \div( vb_\eps) - v\nabla\cdot b_\eps +(b-b_\eps)\cdot \nabla v+\phi'(w)r.
\end{align}
In addition to $D_t$, we set
\begin{align*}
    \Gamma_t &= \partial D_t \cap \{|x|=R(T-t)\}.
\end{align*}
Using \eqref{eq:partial_t_v} and integration by parts, we can compute 
\begin{align}
    \frac{\d}{\d t}J(t)& = \int_{D_t}\partial_t v  - R \int_{\Gamma_t} v\nonumber\\
    & = \int_{\Gamma_t}(\mathbf{n}\cdot b_\eps - R)v+ \int_{\partial D_t \setminus \Gamma_t} (\mathbf{n}\cdot b_\eps)v + \int_{D_t} v(-\nabla \cdot b_\eps) +\int_{D_t}(b-b_\eps)\cdot \nabla v\label{eq:middle_line}\\
    &\qquad\qquad\qquad+\int_{D_t}\phi'(w)r,\nonumber
\end{align}
where $\mathbf{n}$ stands for the outer normal vector, and the integrations are only carried out in the spacial variable. We treat the integrals in \eqref{eq:middle_line} individually. By the definitions of $b_\eps$ and $\xi_\eps$, we can see $|b_\eps|\leq R$. Hence, the first integral is nonpositive. Due to \eqref{eq:b_eps_psd} and the fact that $-\mathbf{n}\in\S^K_+$ on $\partial D_t \setminus \Gamma_t$, the second integral is also nonpositive. In view of \eqref{eq:div_b_eps_geq_0}, the third integral is again nonpositive, while the last one is $o_\eps(1)$. Therefore, taking $\eps\to 0$, we conclude that $\frac{\d }{\d t}J(t)\leq \int_{D_t}\phi'(w)|v|$ as desired.

\end{proof}

\begin{proof}[Proof of Proposition~\ref{Prop:uniq_HJ}]

Let $f$ and $g$ be two weak solutions to \eqref{eq:HJ_eqn} with $f(0,\cdot)=g(0,\cdot)$.
Let $M = \|f\|_{\mathrm{Lip}}\vee \|g\|_{\mathrm{Lip}}$.
For each $\delta>0$, we have $\|f(\delta, \cdot)-g(\delta, \cdot)\|_\infty \leq M\delta$. Let $\phi:\R\to\R_+$ be a smooth function and satisfy 
\begin{align*}
    \begin{cases}
    \phi(z) = 0,&\quad \text{if }|z|\leq M\delta,\\
    \phi(z)>0, &\quad \text{otherwise}.
    \end{cases}
\end{align*}
Applying Lemma~\ref{lemma:weak_comp} to $f,g,M,\phi$ described above, and any choice of $T,\eta$, we have $J(t)\leq J(\delta)$ for $t\in[\delta,T]$. But our choice of $\phi$ implies that
\begin{align*}
    J(\delta) = \int_{D_\delta}\phi(f-g)(\delta,h)\d h = 0.
\end{align*}
Since $J(t)$ is nonnegative, we must have $J(t)=0$ for all $t\in[\delta, T]$. This together with the definition of $\phi$ guarantees that
\begin{align*}
    |f(t,h)-g(t,h)|\leq M\delta, \quad \forall h\in D_t,\ \forall t\in[\delta, T].
\end{align*}
Recall the definition of $D_t$ in \eqref{eq:D_t_uniqueness} which depends on $T$ and $\eta$.
Sending $\delta\to0$, $\eta\to0$ and $T\to \infty$, we conclude that $f=g$.

\end{proof}

\subsection{Assumptions on \texorpdfstring{$\cA_\H$}{AH}}

 Lastly, we show that assumptions on $\cA_\H$ in Theorem~\ref{thm:general_cvg_bF} are satisfied when $\H$ is convex and in the special case considered in Section~\ref{section:special} for $p=1$ or $p$ even.

\begin{Lemma}\label{lemma:H_monotone_special}
If $\H$ is convex, then $\cA_\H$ is convex and contains $\bF_N(t,\cdot)$ for all $t$ and $N$. In the special case where $\H$ is given in \eqref{eq:H_special} and $p=1$ or $p$ is even, we have that $\H$ is convex.
\end{Lemma}

\begin{proof}
Note that, if $\phi:\S^K_+\to\R$ is smooth, then we have
\begin{align*}
    \nabla\cdot \big(\cD\H(\nabla\phi)\big) = \cD^2\H(\nabla\phi)\cdot \nabla^2\phi.
\end{align*}
If $\H$ is convex, a sufficient condition for the above to be nonnegative is the convexity of $\phi$. Recall that convexity is required in the definition of $\cA$ given above \eqref{eq:A_H}. Hence, by regularizing functions in $\cA$, we can see  $\cA_\H=\cA$ when $\H$ is convex. It is also clear that $\cA$ is convex. Due to \eqref{eq:1st_der_bF}, \eqref{eq:1st_der_bF_lower}, and \eqref{eq:2nd_der_bF}, we have $\bF_N(t,\cdot)\in\cA$ for all $t$ and $N$. This completes the proof of the first part of the lemma.

\smallskip

Now, let $\H$ be given in \eqref{eq:H_special}. By computing the limit of $\eps^{-1}(\H(q+\eps a)-\H(q))$, we can see $a\cdot \cD\H(q) = pa\cdot q^{\circ p-1} $ where $\circ$ denotes the Hadamard product. Differentiate one more time to get 
\begin{align}\label{eq:Hessian_H_special}
    a\cdot \cD(a\cdot \cD\H)(q) = p(p-1)(a^{\circ 2})\cdot (q^{\circ p-2})
\end{align}
for all $a\in\S^K$ and $q\in\S^K_+$. If $p=1$ or $p$ is even, this quantity is nonnegative. Hence the convexity of $\H$ follows.
\end{proof}

\section{Convergence to the weak solution}\label{section:cvg_weak}

The goal of this section is to prove Theorem~\ref{thm:general_cvg_bF}. The plan is to first prove the convergence of $\bF_N$ assuming the existence of a weak solution $f$ to \eqref{eq:HJ_eqn} with $f(0,\cdot)=\psi$. Next, we prove the existence of solutions by using a similar argument. We adopt this plan because notation is much simpler in the first part, and the two parts are independent. Theorem~\ref{thm:general_cvg_bF} follows from Proposition~\ref{prop:cvg_assume_existence} and Proposition~\ref{prop:existence} proved in Section~\ref{section:cvg_assume_existence} and Section~\ref{section:existence}, respectively.

\subsection{Convergence when assuming existence of solutions}\label{section:cvg_assume_existence}
Let us assume $f$ is a weak solution to \eqref{eq:HJ_eqn} satisfying $f(0,\cdot)=\psi$. We want to show that $\bF_N$ converges to $f$ as $N\to\infty$. The goal can be summarized as follows.

\begin{Prop}\label{prop:cvg_assume_existence}
In addition to the assumptions in Theorem~\ref{thm:general_cvg_bF}, we assume that there is a unique weak solution $f$ to \eqref{eq:HJ_eqn} with $f(0,\cdot)=\psi$. Then, there is $C>0$ such that \eqref{eq:cvg_rate} holds for all $M\geq 1$ and all $N\in\N$.
\end{Prop}

\smallskip

\begin{proof}

Step 1.
For $N\in\N$, we set
\begin{align}\label{eq:def_r_N}
    r_N = \partial_t\bF_N - \H\big(\nabla \bF_N\big).
\end{align}
For $\delta>0$, define $\phi_\delta:\R\to [0,\infty)$ by
\begin{align}\label{eq:def_phi_delta}
    \phi_\delta(s)=(\delta+s^2)^\frac{1}{2},
\end{align}
which serves as a smooth approximation of the absolute value.
Since $\bF_N$ is Lipschitz uniformly in $N$ due to \eqref{eq:1st_der_bF}, we can set $M =\|f\|_{\mathrm{Lip}}\vee \sup_{N\in\N}\|\bF_N\|_{\mathrm{Lip}}$. Then, we apply Lemma~\ref{lemma:weak_comp} to $\bF_N, f, M, \phi_\delta$, and any choice of $T,\eta$ to see that
\begin{align}\label{eq:dt_J_N}
    \frac{\d}{\d t}J_\delta(t)\leq  \int_{D_t}\phi'_\delta\big(\bF_N -f \big)|r_N|\leq \int_{D_t}|r_N|,\quad\forall t\in[0,T],
\end{align}
where 
\begin{align}\label{eq:def_J(t)}
    J_\delta(t)  = \int_{D_t} \phi_\delta\big(\bF_N-f\big)(t,h)\d h,
\end{align}
for $D_t$ given in \eqref{eq:D_t_uniqueness}. Also recall the definition of $R$ in Lemma~\ref{lemma:weak_comp}.

\medskip

Step 2. We estimate $\int_{D_t}|r_N|$.
Due to the definition of $r_N$ in \eqref{eq:def_r_N}, Proposition~\ref{Prop:approx_HJ} gives an upper bound for $|r_N|^2$. Hence, writing
\begin{align}\label{eq:gamma}
    \gamma = K(K+1)/2,
\end{align}
we have
\begin{align}
    \int_{D_t}|r_N|& \leq |D_t|^\frac{1}{2}\bigg(\int_{D_t} |r_N|^2\bigg)^\frac{1}{2}\nonumber\\
    &\leq CT^{\gamma/2}\bigg(N^{-\frac{1}{4}}\int_{D_t}\kappa(h)\big(\Delta \bF_N+C|h^{-1}|\big)^{\frac{1}{4}}\d h+\int_{D_t}\E\big|\nabla F_N-\nabla \bF_N\big|^2 \d h\bigg)^\frac{1}{2}.\label{eq:int_rn_up_bdd_1}
\end{align}
Here and henceforth, we absorb the constant $R$ in the definition of $D_t$ in \eqref{eq:D_t_uniqueness} into the constant $C$. To bound the first integral in \eqref{eq:int_rn_up_bdd_1}, recall the definition of $\kappa(h)$ in \eqref{eq:def_kappa(h)}, use the definition of $D_t$ and invoke H\"older's inequality to see
\begin{align*}
    \int_{D_t}\kappa(h)\big(\Delta \bF_N+C|h^{-1}|\big)^{\frac{1}{4}}\d h \leq C\eta^{-1} T|D_t|^\frac{3}{4} \bigg(\int_{D_t}\Delta \bF_N+|h^{-1}|\bigg)^\frac{1}{4}.
\end{align*}
In view of \eqref{eq:1st_der_bF}, using integration by parts, we have
\begin{align*}
    \int_{D_t}\Delta\bF_N \leq CT^{\gamma-1}.
\end{align*}
The integral $\int_{D_t}|h^{-1}|$ is bounded by $C\eta^{-1}T^\gamma $. Therefore, we obtain
\begin{align*}
    \int_{D_t}\kappa(h)\big(\Delta \bF_N+C|h^{-1}|\big)^{\frac{1}{4}}\d h \leq C\eta^{-\frac{5}{4}}T^{1+\gamma}.
\end{align*}

To avoid heavy notation, let us write
\begin{align}\label{eq:notation_K,L}
    \bcK=\frac{\cK_{RT,n}}{T^\beta},\qquad \cL=\cL_{\psi,RT,n}.
\end{align}
Here, $\beta$ is given in the assumption of Theorem~\ref{thm:general_cvg_bF}.
For the last integral in \eqref{eq:int_rn_up_bdd_1}, we will show in Step~4 that
\begin{align}\label{eq:derivative_concentration}
    \E\int_{D_t}\big|\nabla( F_N -  \bF_N)\big|^2\leq CT^{\gamma+\beta}\eta^{-\frac{3}{2}}\bcK .
\end{align}
These estimates imply that
\begin{align}\label{eq:int_r_N_bound}
    \int_{D_t}|r_N| \leq CT^{\gamma+\frac{\beta\vee1}{2}}\eta^{-\frac{3}{4}}\big(N^{-\frac{1}{8}}+\bcK^\frac{1}{2}\big).
\end{align}

\medskip

Step 3. We estimate $J_\delta(t)$, extend the integration from over $D_t$ to $\S^K_{+,R(T-t)}$ (defined in \eqref{eq:def_SM}), and conclude the result.
Use \eqref{eq:int_r_N_bound} and \eqref{eq:dt_J_N} to see
\begin{align}\label{eq:J_delta_to_0}
    J_\delta(t) \leq J_\delta(0)+CT^{\alpha}\eta^{-\frac{3}{4}}\big(N^{-\frac{1}{8}}+\bcK^\frac{1}{2}\big),\quad t\in[0,T],
\end{align}
where we set
\begin{align}\label{eq:alpha}
    \alpha = \gamma+\frac{\beta\vee1}{2}+1.
\end{align}
Recall definitions \eqref{eq:def_L}, \eqref{eq:def_phi_delta} and \eqref{eq:def_J(t)}. Hence, for $t=0$, we have 
\begin{align*}
   \lim_{\delta\to0} J_\delta(0)= \int_{D_0}\big|\bF_N(0,h)-f(0,h)\big|\d h \leq CT^\gamma \cL.
\end{align*}
Sending $\delta\to0$ in \eqref{eq:J_delta_to_0} and using the above display, we derive that
\begin{align*}
    \sup_{t\in[0,T]}\int_{D_t}\big|\bF_N(t,h)-f(t,h)\big|\d h\leq C T^{\alpha}\Big(\cL+\eta^{-\frac{3}{4}}\big(N^{-\frac{1}{8}}+\bcK^\frac{1}{2}\big)\Big).
\end{align*}

\smallskip

Due to \eqref{eq:1st_der_bF} and the fact that $\bF_N(0,0)=0$, we have $|\bF_N(t,h)|\leq C(t+|h|)$ uniformly in $N$. By $\bF_N(0,0)=0$ and the assumption on $\psi$ in Theorem~\ref{thm:general_cvg_bF}, we can see $\psi(0)=0$. Since $f(0,\cdot)=\psi$ and the definition of weak solutions requires $f$ to be Lipschitz, we have $|f(t,h)|\leq C(t+|h|)$. In addition, the measure of the set $\S^K_{+,R(T-t)}\setminus D_t$ is bounded by $CT^{\gamma-1}\eta$. Hence, we have
\begin{align*}
    \sup_{t\in[0,T]}\int_{\S^K_{+,R(T-t)}\setminus D_t}\big|\bF_N(t,h) - f(t,h)\big|\d h\leq \sup_{t\in[0,T]}\int_{\S^K_{+,R(T-t)}\setminus D_t}C T \leq CT^\gamma\eta,
\end{align*}
Therefore, we obtain
\begin{align*}
    \sup_{t\in[0,T]}\int_{\S^K_{+,R(T-t)}}\big|\bF_N(t,h)-f(t,h)\big|\d h\leq C T^{\alpha}\Big(\eta+\cL+\eta^{-\frac{3}{4}}\big(N^{-\frac{1}{8}}+\bcK^\frac{1}{2}\big)\Big).
\end{align*}
Let us now specify $T$ and $\eta$. We set $T$ proportional to $M$ to ensure $[0,M]\times \SM\subset \{(t,h):t\in[0,T],\ h\in \S^K_{+,R(T-t)}\}$. Inserting this $T$ and $\eta = (N^{-\frac{1}{8}}+\bcK^\frac{1}{2})^\frac{4}{7}$ into the above display to see
\begin{align}\label{eq:est_cvg}
    \sup_{t\in[0,M]}\int_{\S^K_{+,M}}\big|\bF_N(t,h)-f(t,h)\big|\d h \leq CM^{\alpha}\big(\cL+N^{-\frac{1}{14}}+\bcK^\frac{2}{7}\big).
\end{align}
Recall the notation \eqref{eq:gamma}, \eqref{eq:notation_K,L} and \eqref{eq:alpha}. This gives the desired result \eqref{eq:cvg_rate}.

\medskip

Step 4. To complete the proof, it remains to verify \eqref{eq:derivative_concentration}.
Integrating by parts, we have
\begin{align}
    \int_{D_t}\big|\nabla(F_N &- \bF_N)\big|^2 = \int_{\partial D_t}(F_N-\bF_N)\nabla(F_N-\bF_N)\cdot \mathbf{n} - \int_{D_t}(F_N-\bF_N)\Delta(F_N-\bF_N)\nonumber\\
    &\leq \|F_N-\bF_N\|_{L^\infty([0,RT]\times \S^K_{+,RT})}\bigg(\int_{\partial D_t}\big|\nabla(F_N - \bF_N)\big|+\int_{D_t}\big|\Delta(F_N-\bF_N)\big|\bigg),\label{eq:IBP_step}
\end{align}
Let us estimate the last integral. The lower bound \eqref{eq:2nd_der_bF} shows $\Delta \bF_N\geq 0$, and the lower bound \eqref{eq:2nd_der_F} implies that
\begin{align*}
    \Delta F_N + CN^{-\frac{1}{2}} |Z||h^{-1}|^\frac{3}{2}\geq 0.
\end{align*}
These yield
\begin{align*}
    \int_{D_t}\big|\Delta(F_N-\bF_N)\big|&\leq \int_{D_t} \big|\Delta F_N\big|+\big|\Delta\bF_N\big|\\
    &\leq CT^\gamma N^{-\frac{1}{2}} \eta^{-\frac{3}{2}}|Z|+\int_{D_t} \big(\Delta F_N + \Delta\bF_N\big)  .
\end{align*}
Applying integration by parts to the last integral and using \eqref{eq:1st_der_bF} and \eqref{eq:1st_der_F}, we can see that
\begin{align*}
     \int_{ D_t} \big(\Delta F_N + \Delta\bF_N\big) \leq \int_{\partial D_t}\big|\nabla F_N|+ |\nabla \bF_N|\leq CT^{\gamma-1} \big(1 + N^{-\frac{1}{2}}\eta^{-\frac{1}{2}}|Z|\big).
\end{align*}

This display also serves as a bound for the first integral in \eqref{eq:IBP_step}. Insert the above two displays into \eqref{eq:IBP_step} to get
\begin{align*}
    \int_{D_t}\big|\nabla(F_N - \bF_N)\big|^2\leq CT^\gamma\|F_N-\bF_N\|_{L^\infty([0,RT]\times \S^K_{+,RT})} \eta^{-\frac{3}{2}}\Big (1+N^{-\frac{1}{2}}|Z|\Big).
\end{align*}
Recall \eqref{eq:def_K} and \eqref{eq:notation_K,L}. Take expectations on both sides of this inequality and invoke the Cauchy--Schwarz inequality to conclude \eqref{eq:derivative_concentration}.

\end{proof}

\subsection{Existence of weak solutions}\label{section:existence}

To complete the proof of Theorem~\ref{thm:general_cvg_bF}, we need the following existence result.

\begin{Prop}\label{prop:existence}
Under the assumptions in Theorem~\ref{thm:general_cvg_bF}, there is a unique weak solution $f$ to \eqref{eq:HJ_eqn} with $f(0,\cdot)=\psi$.
\end{Prop}

\begin{proof}

The uniqueness part follows from Proposition~\ref{Prop:uniq_HJ}. Hence, we only need to prove the existence.
We first show that $(\bF_N)_{N\in\N}$ is a Cauchy sequence in the local uniform topology and then verify that the limit is a weak solution. 

\medskip

Step 1. We show that the sequence $(\bF_N)_{N\in\N}$ is Cauchy. We proceed similarly as in the previous subsection. Recall the definition of $r_N$ in \eqref{eq:def_r_N} and $\phi_\delta$ in \eqref{eq:def_phi_delta}. Let $N,N'\in\N$. Now, setting $M = \sup_{N\in\N}\|\bF_N\|_{\mathrm{Lip}}$ and applying Lemma~\ref{lemma:weak_comp} to $\bF_N$ and $\bF_{N'}$, we obtain
\begin{align*}
    \frac{\d}{\d t}J_\delta(t) \leq \int_{D_t}\phi'_\delta\big(\bF_N-\bF_{N'}\big)|r_N-r_{N'}|\leq \int_{D_t}|r_N|+|r_{N'}|
\end{align*}
where 
\begin{align*}
    J_\delta(t) = \int_{D_t}\phi_\delta\big(\bF_N-\bF_{N'}\big)(t,h)\d h.
\end{align*}

The rest follows exactly the same procedure after \eqref{eq:def_J(t)} in the previous section. The only difference is that we have more terms due to the presence of $\bF_{N'}$, but they are treated in the same way as for $\bF_N$. Similar to \eqref{eq:est_cvg}, one can see that eventually we obtain
\begin{align*}
    \sup_{t\in[0,M]}\int_{S^K_{+,M}}\big|\bF_N(t,h)-\bF_{N'}(t,h)\big|\d h \leq CM^\alpha\Big(\cL_{\psi,CM,N}+N^{-\frac{1}{14}}+ (\cK_{CM,N}/M^\beta)^\frac{2}{7}\\
      +\cL_{\psi,CM,N'}+ {N'}^{-\frac{1}{14}}+(\cK_{CM,N'}/M^\beta)^\frac{2}{7}\Big).
\end{align*}
Hence, by the assumption of Theorem~\ref{thm:general_cvg_bF} on the decay of $\cK_{M,N}$ and $\cL_{\psi,M,N}$, we know that $(\bF_N)_{N\in\N}$ is Cauchy in local $L^\infty_tL^1_h$. Due to the argument in Remark~\ref{Rem:infinity_norm}, we can upgrade this to $(\bF_N)_{N\in\N}$ being Cauchy in local $L^\infty_tL^\infty_h$. Let us denote the limit by $f$.

\bigskip

Step 2. We verify that $f$ is a weak condition by checking that each property listed in Definition~\ref{Def:weak_sol} is satisfied by $f$ and that $f(0,\cdot)=\psi$.

\smallskip

Firstly, we verify that $f$ is Lipschitz and satisfies the initial condition.
Since $\bF_N$ is Lipschitz uniformly in $N$ due to \eqref{eq:1st_der_bF}, we can conclude that $f$ is Lipschitz. Due to the assumption $\lim_{N\to\infty}L_{\psi,M,N}=0$, we have $f(0,\cdot)=\psi$.

\smallskip

Next, we show that $f(t,\cdot)\in\cA_\H$ for every $t\geq 0$. By \eqref{eq:2nd_dt_bF} and \eqref{eq:2nd_der_bF}, we have that both $\bF_N$ and $f$ are convex in the temporal variable and convex in the spacial variable. It is well known that convexity implies convergence of derivatives at each point of differentiability. The Lipschitzness of $f$ and Rademacher's theorem imply that $f$ is differentiable almost everywhere (a.e.). Hence, we can deduce that $(\partial_t, \nabla)\bF_N$ converges to $(\partial_t, \nabla)f$ pointwise a.e. Since $\bF_N(t,\cdot) \in \cA_\H$ for every $t$ and $N$, the claim can be verified by passing to the limit.

\smallskip

Lastly, we show that $f$ satisfies \eqref{eq:HJ_eqn} a.e.
Since $\bF_N$ is Lipschitz uniformly in $N$ due to \eqref{eq:1st_der_bF} and $\H$ is continuous, the bounded convergence theorem implies that, for any compact $B\subset \S^K_{++}$ and $t$ a.e., 
\begin{align*}
    \int_B\Big|\partial_t f - \H(\nabla f)\Big|(t,h)\d h = \lim_{n\to\infty}\int_B\Big|\partial_t \bF_N - \H(\nabla \bF_N)\Big|(t,h)\d h.
\end{align*}
We want to show that the right hand side is zero. Recall the definition of $D_t$ in \eqref{eq:D_t_uniqueness}. By choosing $T$ and $\delta $ in $D_t$ suitably, we can ensure $B\subset D_t$. Then, by \eqref{eq:def_r_N}, \eqref{eq:int_r_N_bound} and the assumption $\lim_{N\to\infty}\cK_{M,N}=0$ in the statement of Theorem~\ref{thm:general_cvg_bF}, we conclude that the right hand side of the above display is zero. Since $B$ and $t$ are arbitrary, we conclude that $\partial_t f - \H(\nabla f)=0$ a.e.

\end{proof}

\section{Viscosity solutions of Hamilton--Jacobi equations}\label{section:vis_sol}

In this section, we give the precise definition of viscosity solutions. After that, we prove the comparison principle which ensures the uniqueness of solutions. In addition, we verify that the Hopf formula is a solution. Classical references include \cite{evans2010partial,crandall1992user}. See also \cite{bardi1984hopf,lions1986hopf}. Here, we follow the approach in \cite{mourrat2020nonconvex}.

\smallskip

A function $f:\R_+\times\S^K_+\to\R$ is said to be nondecreasing if $f(t,x)- f(t',x')\geq 0$ whenever $t\geq t'$ and $x-x'\in\S^K_+$. A function $\psi:\S^K_+\to\R$ is said to be nondecreasing if $\psi(x)-\psi(x')\geq0$ whenever $x-x'\in\S^K_+$.

\begin{Def}\label{def:vs}
\leavevmode
\begin{enumerate}
    \item \label{item:vs_1} A nondecreasing continuous function $f:\R_+\times \S^K_+\to \R$ is a viscosity subsolution to \eqref{eq:HJ_eqn} if for every $(t,x) \in (0,\infty)\times \S^K_+$ and every smooth $\phi:(0,\infty)\times \S^K_+$ such that $f-\phi$ has a local maximum at $(t,x)$, we have
\begin{align*}
    \begin{cases}
    \big(\partial_t \phi - \H(\nabla\phi)\big)(t,x)\leq 0, \quad &\text{if }x\in \S^K_{++},\\
     \nabla\phi(t,x) \in\S^K_+,  \quad &\text{if }x\in \S^K_+\setminus\S^K_{++}.
    \end{cases}
\end{align*}

\item \label{item:vs_2} A nondecreasing continuous function $f:\R_+\times \S^K_+\to \R$ is a viscosity supersolution to \eqref{eq:HJ_eqn} if for every $(t,x) \in (0,\infty)\times \S^K_+$ and every smooth $\phi:(0,\infty)\times \S^K_+$ such that $f-\phi$ has a local minimum at $(t,x)$, we have
\begin{align*}
    \begin{cases}
    \big(\partial_t \phi - \H(\nabla\phi)\big)(t,x)\geq 0, \quad &\text{if }x\in \S^K_{++},\\
    \partial_t \phi (t,x) - \inf \H(q)\geq 0,\quad ,\quad &\text{if }x\in \S^K_+\setminus\S^K_{++},
    \end{cases}
\end{align*}
where the infimum is taken over all $q\in \big(\nabla \phi(t,x) +\S^K_+\big)\cap \S^K_+  $ and $  |q|\leq \|f\|_{\mathrm{Lip}}$.

\item A nondecreasing continuous function $f:\R_+\times \S^K_+\to \R$ is a viscosity solution to \eqref{eq:HJ_eqn} if $f$ is both a viscosity subsolution and supersolution.

\end{enumerate}

\end{Def}

\medskip

\begin{Rem}

The restriction $|q|\leq \|f\|_{\mathrm{Lip}}$ under the infimum in Definition~\ref{def:vs}~\eqref{item:vs_2} can be replaced by $|q|\leq \|f\|_{\mathrm{Lip}}+c$ for any $c\geq 0$. 
Indeed, since $f$ is assumed to be Lipschitz, we can always restrict $\H$ to the set $\{q\in\S^K_+:\ |q|\leq c'\}$ without altering the equation \eqref{eq:HJ_eqn} as long as $c'\geq \|f\|_{\mathrm{Lip}}$. Aside from this heuristic, one can straightforwardly check that the choice of $c$ does not affect the results in this and the next sections.

\end{Rem}

\begin{Rem}
The only properties of $\H$ used in this section are the positivity $\H\geq 0$, local Lipschitzness as in \eqref{eq:H(a)-H(b)} and nondecreasingness given by Lemma~\ref{lemma:DH_psd}. The following two propositions are still valid for general $\H$ satisfying these three properties.
\end{Rem}

\medskip

\begin{Prop}[Comparison principle]\label{prop:comp_principle}
Let $u$ be a subsolution and $v$ be a supersolution of \eqref{eq:HJ_eqn}. Assume $u$ and $v$ are Lipschitz. Then, we have
\begin{align*}
    \sup_{\R_+\times \S^K_+} (u-v) =\sup_{\{0\}\times\S^K_+}(u-v).
\end{align*}
\end{Prop}

\begin{Prop}[Hopf formula]\label{Prop:hopf_viscosity}
Suppose $\psi:\S^K_+\to\R$ is convex, Lipschitz and nondecreasing. Let $f$ be given in \eqref{eq:f_Hopf_0}. Then $f$ is Lipschitz and is a viscosity solution to \eqref{eq:HJ_eqn} with initial condition $f(0,\cdot)=\psi$.
\end{Prop}

\subsection{Proof of Proposition~\ref{prop:comp_principle}}
Let us argue by contradiction and assume
\begin{align}\label{eq:contradiction}
    \sup_{\R_+\times \S^K_+} (u-v) >\sup_{\{0\}\times\S^K_+}(u-v).
\end{align}
We start by modifying $u$. For $\eps \in(0,1)$ to be specified later, we set
\begin{align*}
    u_\eps (t,x) = u(t,x)+\eps \tr(x)-C\eps t,\quad \forall (t,x)\in\R_+\times \S^K_+
\end{align*}
where $\tr$ stands for the trace. Let $I$ be the $K\times K$ identity matrix. By choosing $C$ large and then $\eps$ small, we can ensure that, if $u_\eps-\phi$ attains a local maximum at $(t,x)$, we have
\begin{align}\label{eq:u_eps_eqn}
    \begin{cases}
    (\partial_t \phi - \H(\nabla\phi ))(t,x)\leq -2\eps, \quad &\textrm{if } x\in\S^K_{++}
    \\(\nabla \phi - \eps I)(t,x)\in\S^K_+,\quad &\textrm{if } x\in \S^K_+\setminus\S^K_{++}.
    \end{cases}
\end{align}
Since $u(t,\cdot)$ in nondecreasing for each $t$, we also have
\begin{align}\label{eq:u_eps_lower_bd}
    u_\eps(t,x+y) - u_\eps(t,x)\geq \eps\tr(y),\quad y \in \S^K_+.
\end{align}
With $\epsilon$ sufficiently small chosen, \eqref{eq:contradiction} still holds with $u$ replaced by $u_\eps$. Next, we replace $u_\eps$ by $u_\eps-\frac{\delta}{T-t}$, where $\delta$ is chosen small enough and $T>1$ is chosen large enough by \eqref{eq:contradiction} to ensure that
\begin{align}\label{eq:sup_[0,T)}
    \sup_{[0,T)\times \S^K_+} (u_\eps-v) >\sup_{\{0\}\times\S^K_+}(u_\eps-v).
\end{align}
Also, note that \eqref{eq:u_eps_eqn} still holds.
In addition, we have, for every $M>0$,
\begin{align}\label{eq:-infty_at_T}
    \lim_{\eta\to 0}\sup_{[T-\eta,T)\times \S^K_{+,M}}u_\eps = -\infty.
\end{align}

\smallskip

Next, we introduce some parameters and auxiliary functions. By the formula for $\H$ in \eqref{eq:H}, there is a constant $C_\H$ such that 
\begin{align}\label{eq:H(a)-H(b)}
    \big|\H(a)- \H(b)\big|\leq C_\H|a-b|(|a|+|b|)^{p-1},\quad \forall a, b\in \S^K_+.
\end{align}
Let
\begin{align}\label{eq:K_L}
    L = 1+\|u\|_{\mathrm{Lip}}+\|v\|_\mathrm{Lip},\qquad K = C_\H(4L)^{p-1}.
\end{align}
Due to the definition of $u_\eps$, the following holds for all $(t,x)\in[0,T)\times \S^K_+$,
\begin{align}\label{eq:u_eps<L|x|}
    u_\eps(t,x) \leq C+L|x|,\qquad 
    \|\nabla u_\eps(t,x)\|\leq L.
\end{align}
By \eqref{eq:sup_[0,T)}, there is $(\bar t, \bar x)$ such that
\begin{align}\label{eq:bar_x}
    \big(u_\eps-v\big)(\bar t, \bar x ) >\sup_{\{0\}\times\S^K_+}(u_\eps-v).
\end{align}
Let us set
\begin{align*}
    R = \big(|\bar x|^2+1\big)^\frac{1}{2}+ K\bar t.
\end{align*}
Take $\chi:\R\to \R_+$ to be a smooth function satisfying
\begin{align}\label{eq:chi_prop}
    (r-1)_+ \leq \chi(r)\leq r_+,\quad |\chi'(r)|\leq 1, \quad \forall r \in\R,
\end{align}
where the positive sign in the subscript indicates taking the positive part. The function $\chi$ can be viewed as a smoothed version of $r\mapsto r_+$.
Define $\eta:[0,T)\times\S^K_+\to \R$ by 
\begin{align}\label{eq:eta(t,x)}
    \eta(t,x) = 2L\chi \Big((|x|^2+1)^\frac{1}{2} + K t - R\Big).
\end{align}
We claim
\begin{align}\label{eq:u-v-eta}
    \sup_{[0,T)\times \S^K_+} (u_\eps-v-\eta) =\sup_{\{0\}\times\S^K_+}(u_\eps-v-\eta).
\end{align}
On the other hand, due to \eqref{eq:bar_x} and the definitions of $R$ and $\eta$, we have 
\begin{align*}
    \sup_{[0,T)\times \S^K_+} (u_\eps-v-\eta)\geq \big(u_\eps-v\big)(\bar t, \bar x )>\sup_{\{0\}\times\S^K_+}(u_\eps-v)\geq \sup_{\{0\}\times\S^K_+}(u_\eps-v-\eta),
\end{align*}
which contradicts \eqref{eq:u-v-eta}. Hence, the proof is complete once the claim \eqref{eq:u-v-eta} is verified.

\subsubsection{Proof of \eqref{eq:u-v-eta}}
Again we argue by contradiction and assume
\begin{align}\label{eq:contradiction_2}
    \sup_{[0,T)\times \S^K_+} (u_\eps-v-\eta) >\sup_{\{0\}\times\S^K_+}(u_\eps-v-\eta).
\end{align}
We are going to employ the classical trick of ``doubling the variables''. For $\alpha \in (0,1)$, we introduce
\begin{align*}
    \Psi_\alpha (t,x,t,x')=u_\eps(t,x) - v(t',x')- \phi_\alpha(t,x,t',x'),\quad \forall t\in[0,T),\ t'>0,\ x,x'\in\S^K_+.
\end{align*}
where
\begin{align*}
    \phi_\alpha(t,x,t',x')= \frac{1}{2\alpha}\big(|t-t'|^2+|x-x'|^2\big)+\eta(t,x).
\end{align*}

Step 1. We show that there exists a maximizer $(t_\alpha, x_\alpha,t'_\alpha,x'_\alpha)$ of $\Psi_\alpha$, and they converge as $\alpha \to 0$. To start, we seek an upper bound for $\Psi_\alpha$. The nondecreasingness of $v$ gives $-v(t,x)\leq -v(0,0)$. The definition of $\eta$ in \eqref{eq:eta(t,x)} shows $\eta(t,x)\geq 2L( |x|+Kt - R-1)$. Using these and the first inequality in \eqref{eq:u_eps<L|x|}, we have
\begin{align*}
    \Psi_\alpha(t,x,t',x') \leq C-L|x| - \frac{1}{2\alpha}\big(|t-t'|^2+|x-x'|^2\big) - 2KLt.
\end{align*}
Here and henceforth, we absorb $L$, $K$ and $R$ into $C$. Now, one can see the existence of a maximizer $(t_\alpha, x_\alpha,t'_\alpha,x'_\alpha)$. Then, we have
\begin{align*}
    \Psi_\alpha (t_\alpha, x_\alpha, t'_\alpha, x'
_\alpha) \geq \Psi_\alpha(0,0,0,0)= u_\eps(0,0)-v(0,0).
\end{align*}
Combine the above two displays to see that, for all $\alpha<1$, these points $(t_\alpha, x_\alpha,t'_\alpha,x'_\alpha)$ lie in a bounded set and
\begin{align*}
        |t_\alpha - t'_\alpha |^2 + |x_\alpha - x'_\alpha|^2 \leq C\alpha.
\end{align*}
By passing to a subsequence, we can assume there is $t_0$ and $x_0$ such that $t_\alpha, t'_\alpha \to t_0$ and $x_\alpha,x'_\alpha \to x_0$ as $\alpha \to 0$.

\smallskip
In view of \eqref{eq:-infty_at_T}, we must have $t_0<T$. The maximality of $(t_\alpha,x_\alpha,t'_\alpha,x'_\alpha)$ yields
\begin{align*}
    &\big(u_\eps-v-\eta\big)(t_0,x_0)\leq \sup_{[0,T)\times \S^K_+}(u_\eps-v-\eta )\\
    &\quad \leq \Psi_\alpha(t_\alpha,x_\alpha,t'_\alpha,x'_\alpha)\leq u_\eps(t_\alpha,x_\alpha)-v(t'_\alpha,x'_\alpha)-\eta(t_\alpha,x_\alpha).
\end{align*}
Take $\alpha\to0$ and use the continuity of $u_\eps$, $v$ and $\eta$ to see 
\begin{align*}
    \big(u_\eps-v-\eta\big)(t_0,x_0) = \sup_{[0,T)\times \S^K_+}(u_\eps-v-\eta).
\end{align*}
By \eqref{eq:contradiction_2}, we must have $t_0> 0$. Henceforth, we fix a sufficiently small $\alpha$ so that $t_\alpha,t'_\alpha>0$.

\bigskip

Step 2. For this fixed $\alpha$, note that
\begin{align}\label{eq:(t,x)_mapsto}
    (t,x)\mapsto u_\eps(t,x) - v(t'_\alpha,x'_\alpha) -\phi_\alpha(t,x,t'_\alpha, x'_\alpha)
\end{align}
has a local maximum at $(t_\alpha, x_\alpha)$. We argue that
\begin{align}\label{eq:x_alpha_pd}
    x_\alpha \in \S^K_{++}.
\end{align}
Otherwise, there is $y\in \S^K_+$ with $|y|=1$ such that
\begin{align}\label{eq:y.x_alpha=0}
    y\cdot x_\alpha = 0.
\end{align}
Under this assumption, we want to derive a contradiction to the fact that the maximum is achieved $(t_\alpha, x_\alpha)$. For $\delta >0$, using \eqref{eq:u_eps_lower_bd}, we can see
\begin{align}
    u_\eps(t_\alpha, x_\alpha+\delta y) - \phi_\alpha(t_\alpha, x_\alpha+\delta y,t'_\alpha, x'_\alpha) - \Big(u_\eps(t_\alpha, x_\alpha) - \phi_\alpha(t_\alpha, x_\alpha,t'_\alpha, x'_\alpha)\Big)\nonumber\\
    \geq \eps\delta \tr(y)-\frac{1}{2\alpha}\Big(2\delta y\cdot(x_\alpha -x_\alpha') +\delta^2\Big)-\Big(\eta(t_\alpha,x_\alpha+\delta y ) - \eta(t_\alpha,x_\alpha) \Big).\label{eq:u-phi-(u-phi)}
\end{align}
The definition of $\eta$ in \eqref{eq:eta(t,x)} allows us to compute
\begin{align}\label{eq:grad_eta}
     \nabla\eta(t,x) = \frac{x}{(|x|^2+1)^\frac{1}{2}}2L\chi'\bigg((|x|^2+1)^\frac{1}{2}+Lt-R\bigg).
\end{align}
By \eqref{eq:y.x_alpha=0}, we have $y\cdot \nabla\eta(t_\alpha,x_\alpha)=0$. This along with Taylor's theorem implies
\begin{align*}
    \eta(t_\alpha,x_\alpha+\delta y ) - \eta(t_\alpha,x_\alpha)=\mathcal{O}(\delta^2).
\end{align*}
Apply this, \eqref{eq:y.x_alpha=0} and $y\cdot x'_\alpha \geq 0$ to see that \eqref{eq:u-phi-(u-phi)} is bounded below by 
\begin{align*}
     \eps \delta\tr(y) - \mathcal{O}(\delta^2).
\end{align*}
Since $\eps>0$ and $\tr(y)>0$, this is strictly positive for $\delta$ small. This contradicts the fact that \eqref{eq:(t,x)_mapsto} achieves a local maximum at $(t_\alpha, x_\alpha)$. By contradiction, we must have \eqref{eq:x_alpha_pd}.

\smallskip

Using this, \eqref{eq:u_eps_eqn}, and the maximality of \eqref{eq:(t,x)_mapsto} at $(t_\alpha,x_\alpha)$, we obtain
\begin{gather}
    \frac{1}{\alpha}(t_\alpha - t'_\alpha) +\partial_t\eta(t_\alpha, x_\alpha) - \H\bigg(\frac{1}{\alpha}(x_\alpha-x'_\alpha)+\nabla\eta(t_\alpha, x_\alpha)\bigg)\leq -2\eps.\label{eq:test_<0}
\end{gather}

\bigskip

Step 3. Still for this fixed $\alpha$, the function
\begin{align*}
    (t',x')\mapsto v(t',x') - u_\eps(t_\alpha,x_\alpha) + \phi_\alpha(t_\alpha,x_\alpha,t',x')
\end{align*}
attains a local minimum at $(t'_\alpha ,x'_\alpha)$. Note that $-\nabla_{x'}\phi_\alpha(t_\alpha,x_\alpha,t',x')=\frac{1}{\alpha}(x_\alpha -x')$. We claim that there is $a \in \S^K_+$ such that
\begin{gather}
    a - \frac{1}{\alpha}(x_\alpha - x'_\alpha) \in \S^K_+,\label{eq:a_in_x+S^K}\\
    |a|\leq \|v\|_\mathrm{Lip},\label{eq:|a|<|v|}\\
    \frac{1}{\alpha}(t_\alpha-t'_\alpha) -\H(a)\geq -\eps.\label{eq:test>0}
\end{gather}

If $x'_\alpha\in \S^K_{++}$, then by setting $a = \frac{1}{\alpha}(x_\alpha - x'_\alpha)$, we clearly have \eqref{eq:a_in_x+S^K}. In this case, the local minimum is achieved at an interior point $x'_\alpha$. Since $v$ is nondecreasing, we can see $\frac{1}{\alpha}(x_\alpha-x'_\alpha)\in\S^K_+$ and thus $a\in\S^K_+$. Then \eqref{eq:test>0} follows from the definition of supersolutions. If $v(t'_\alpha,\cdot)$ is differentiable at $x'_\alpha$, then the minimality at $x'_\alpha$ implies $\frac{1}{\alpha}(x_\alpha - x'_\alpha) = \nabla v(t'_\alpha,x'_\alpha)$ and hence \eqref{eq:|a|<|v|} holds. If $x'_\alpha$ is not a point of differentiability, then \eqref{eq:|a|<|v|} still holds by a regularizing argument.

\smallskip

If $x'_\alpha\not\in \S^K_{++}$, namely $x'_\alpha\in\S^K_+\setminus \S^K_{++}$, then the existence of $a$ and \eqref{eq:a_in_x+S^K}--\eqref{eq:test>0} directly follow from the boundary condition in the definition of supersolutions.

\bigskip

Step 4. We compare \eqref{eq:test_<0} with \eqref{eq:test>0} to derive a contradiction. To start, we derive some estimates. For simplicity, we write
\begin{align*}
    b =\nabla\eta(t_\alpha, x_\alpha).
\end{align*}
Recall the definition of the constant $L$ in \eqref{eq:K_L}. Due to \eqref{eq:grad_eta} and the second inequality in \eqref{eq:chi_prop}, we get $|b|\leq 2L$. By \eqref{eq:|a|<|v|}, we have $|a|\leq L$. These along with \eqref{eq:H(a)-H(b)} yield
\begin{align*}
    \big|\H(a+b)-\H(a)\big|\leq C_\H|b|(4L)^{p-1}.
\end{align*}
Using the definition of $\eta$ in \eqref{eq:eta(t,x)}, we can see
\begin{align*}
    \partial_t\eta(t_\alpha, x_\alpha) \geq K|\nabla\eta(t_\alpha, x_\alpha)| = K|b|.
\end{align*}
The above two displays together with the definition of $K$ in \eqref{eq:K_L} imply 
\begin{align}\label{eq:>0}
    \partial_t\eta(t_\alpha, x_\alpha) -\H(a+b)+\H(a)\geq 0.
\end{align}
On the other hand, from \eqref{eq:test_<0} and \eqref{eq:a_in_x+S^K}, using the monotonicity of $\H $ in Lemma~\ref{lemma:DH_psd}, we have
\begin{align*}
    \frac{1}{\alpha}(t_\alpha - t'_\alpha) +\partial_t\eta(t_\alpha, x_\alpha) - \H(a+b)\leq -2\eps.
\end{align*}
Subtract \eqref{eq:test>0} from the above display to obtain
\begin{align*}
    \partial_t\eta(t_\alpha, x_\alpha) -\H(a+b)+\H(a) \leq -\eps.
\end{align*}
This contradicts \eqref{eq:>0} and thus the proof of \eqref{eq:u-v-eta} is complete.

\subsection{Proof of Proposition~\ref{Prop:hopf_viscosity}}
Let us rewrite the Hopf formula \eqref{eq:f_Hopf_0} as
\begin{align}
    f(t,x) 
    & = \sup_{z\in \S^K_+}\{z\cdot x - \psi^*(z)+t\H(z)\}\label{eq:f_Hopf_1}\\
    & = (\psi^*-t\H)^*(x).\label{eq:f_Hopf_2}
\end{align}
Here the superscript $*$ denotes the Fenchel transformation over $\S^K_+$, namely,
\begin{align}\label{eq:def_u*}
    u^*(x) = \sup_{y\in \S^K_+}\{y\cdot x - u(y)\},\quad \forall x\in \S^K_+.
\end{align}

We check the following in order: nondecreasingness, initial condition, semigroup property (or dynamic programming principle), Lipschitzness, $f$ being a subsolution, and $f$ being a supersolution.

\subsubsection{Nondecreasingness}\label{section:nondecreas}
Since the supremum in \eqref{eq:f_Hopf_1} is taken over $\S^K_+$, it is clear from Lemma~\ref{lemma:psd} that $f(t,\cdot) $ is nondecreasing. By the formula of $\H$ in \eqref{eq:H} and the Schur product theorem, we have $\H(z)\geq 0$ for all $z\in\S^K_+$. Hence, from the formula \eqref{eq:f_Hopf_0}, we can see $f$ is also nondecreasing in $t$.

\subsubsection{Verification of the initial condition}
The desired identity
\begin{align*}
    \psi(x) = \sup_{z\in \S^K_+}\inf_{y\in \S^K_+}\big\{z\cdot (x-y) +\psi(x) \big\}=\psi^{**}(x),\quad \forall x\in\S^K_+.
\end{align*}
follows from a version of Fenchel--Moreau identity stated in Proposition~\ref{prop:biconjugate}

\subsubsection{Semigroup property}

Let $f$ be given in \eqref{eq:f_Hopf_1}. We want to show, for all $s\geq 0$,
\begin{align*}
    f(t+s,x) &= \sup_{z\in\S^K_+}\inf_{y\in\S^K_+}\big\{z\cdot(x-y)+f(t,y) +s\H(z)\big\},
\end{align*}
or, in a more compact form,
\begin{align}\label{eq:semigroup}
    f(t+s,\cdot)= \big(f^*(t,\cdot)-s\H\big)^*.
\end{align}
In view of the Hopf formula \eqref{eq:f_Hopf_2}, this is equivalent to
\begin{align}\label{eq:semigroup_property}
    \big(\psi^* - (t+s)\H\big)^* = \big((\psi^*-t\H)^{**}-s\H\big)^*.
\end{align}
From the definition of the Fenchel transform \eqref{eq:def_u*}, it can be seen that, for any $u$, 
\begin{align}\label{eq:u**leq_u}
    u^{**}\leq u.
\end{align}
Since the Fenchel transform is order-reversing, \eqref{eq:u**leq_u} implies that
\begin{align}\label{eq:semigruop_one_inequality}
    \big((\psi^*-t\H)^{**}-s\H\big)^*\geq \big(\psi^* - (t+s)\H\big)^* .
\end{align}

\smallskip

To see the other direction, we use \eqref{eq:u**leq_u} to get
\begin{align*}
    \frac{s}{t+s}\psi^* + \frac{t}{t+s}\big(\psi^*-(t+s)\H\big)^{**}\leq \psi^*-t\H.
\end{align*}
For any $u$, it can be readily checked that $u^*$ is convex and lower semi-continuous. Using the argument in Section~\ref{section:nondecreas}, we can deduce that $u^*$ is non-decreasing. Hence the left hand side of the above display satisfies the condition in Proposition~\ref{prop:biconjugate}. Therefore, taking the Fenchel transform twice in the above display and applying Proposition~\ref{prop:biconjugate}, we have
\begin{align*}
    \frac{s}{t+s}\psi^* + \frac{t}{t+s}\big(\psi^*-(t+s)\H\big)^{**}\leq (\psi^*-t\H)^{**}.
\end{align*}
Reorder terms and then use \eqref{eq:u**leq_u} to see
\begin{align*}
    \big(\psi^*-(t+s)\H\big)^{**} - (\psi^*-t\H)^{**}\leq \frac{s}{t}\Big((\psi^*-t\H)^{**}-\psi^*\Big)\leq -s\H.
\end{align*}
This immediately gives
\begin{align*}
    \big(\psi^*-(t+s)\H\big)^{**}\leq (\psi^*-t\H)^{**}-s\H.
\end{align*}
Taking the Fenchel transform on both sides and invoking Proposition~\ref{prop:biconjugate}, we have
\begin{align*}
     \big(\psi^* - (t+s)\H\big)^*\geq  \big((\psi^*-t\H)^{**}-s\H\big)^*.
\end{align*}
Here, we also used the order-reversing property of the Fenchel transform. 
This together with \eqref{eq:semigruop_one_inequality} verifies \eqref{eq:semigroup_property}.

\subsubsection{Lipschitzness}

Since $\psi$ is Lipschitz, we have $\psi^*(z)=\infty$ outside the compact set $\{z\in\S^K_+:|z|\leq \|\psi\|_{\mathrm{Lip}}\}$. This together with \eqref{eq:f_Hopf_1} implies that for each $x\in\S^K_+$, there is $z\in \S^K_+$ with $|z|\leq \|\psi\|_\mathrm{Lip}$ such that
\begin{align*}
    f(t,x) = z\cdot x - \psi^*(z)+t\H(z).
\end{align*}
This yields that, for any $x'\in\S^K_+$,
\begin{align*}
    f(t,x) -f(t,x')\leq z\cdot(x-x') \leq \|\psi\|_\mathrm{Lip}|x-x'|.
\end{align*}
By symmetry, we conclude that $f$ is Lipschitz in $x$, and the Lipschitz coefficient is uniform in $t$.

To show the Lipschitzness in $t$, we fix any $x\in\S^K_+$. Then, we have, for some $z\in\S^K_+$ with $|z|\leq \|\psi\|_\mathrm{Lip}$,
\begin{align*}
    f(t,x)&=z\cdot x-\psi^*(z)+t\H(z)\leq f(t',x) + (t-t')\H(z)\\
    &\leq f(t',x) +|t'-t|\bigg(\sup_{|z|\leq \|\psi\|_\mathrm{Lip}}|\H(z)|\bigg).
\end{align*}
Again by symmetry, the Lipschitzness in $t$ is obtained, and its coefficient is independent of $x$.

\subsubsection{The Hopf formula is a subsolution}

Let $\phi:(0,\infty)\times \S^K_+\to\R$ be smooth. Suppose $f-\phi$ achieves a local maximum at $(t,x)\in (0,\infty)\times \S^K_+$. Since $\psi$ is Lipschitz, we can see $\psi^*$ is infinite outside a compact set. Hence, by \eqref{eq:f_Hopf_1}, there is $\bar z\in\S^K_+$ such that
\begin{align*}
    f(t,x) = \bar z\cdot x - \psi^*(\bar z) + t\H(\bar z).
\end{align*}

For the case $x \in \S^K_{++}$, by \eqref{eq:f_Hopf_1}, we have, for $s\in[ 0,t]$ and $h\in \S^K$ sufficiently small,
\begin{align*}
    f(t,x) \leq f(t-s,x+h)-\bar z\cdot h+s\H(\bar z).
\end{align*}
By the assumption on $\phi$, we have
\begin{align*}
    f(t-s,x+h)-\phi(t-s,x+h)\leq f(t,x)-\phi(t,x).
\end{align*}
for small $s\in[0,t]$ and small $h\in\S^K$. Combine the above two inequalities to get
\begin{align}\label{eq:phi(t,x)_upper_bound_verify_HJ}
    \phi(t,x)-\phi(t-s,x+h) \leq -\bar z\cdot h+s\H(\bar z).
\end{align}
Set $s=0$ and vary $h$ to see
\begin{align}\label{eq:z=grad_f_vis}
    \bar z = \nabla \phi(t,x).
\end{align}
Then, we set $h=0$ in \eqref{eq:phi(t,x)_upper_bound_verify_HJ}, take $s\to 0$ and insert \eqref{eq:z=grad_f_vis} to obtain
\begin{align*}
    \partial_t \phi(t,x)-\H(\nabla \phi(t,x))\leq 0.
\end{align*}

If $x\in \S^K_+\setminus\S^K_{++}$, then \eqref{eq:phi(t,x)_upper_bound_verify_HJ} still holds for $h\in \S^K_+$. Set $s=0$ and vary $h$, we can see $a\cdot \nabla \phi(t,x)\geq a\cdot \bar z$ for all $a\in\S^K_+$. Since $\bar z\in \S^K_+$, Lemma~\ref{lemma:psd} implies that $ \nabla \phi(t,x)\in\S^K_+$.

\subsubsection{The Hopf formula is a supersolution}

The idea of proof in this part can be seen in \cite[Proof of Proposition 1]{lions1986hopf}. Let $(t,h)\in(0,\infty)\times \S^K_+$ be a local minimum point for $f-\phi$. Due to \eqref{eq:f_Hopf_1}, $f$ is convex in both variables. Since $\S^K_+$ is also convex, we have, for all $(t',x')\in (0,\infty)\times \S^K_+$ and all $\lambda \in[ 0,1]$,
\begin{align*}
    f(t',x')-f(t,x)\geq \frac{1}{\lambda}\Big(f\big(t+\lambda(t'-t),x+\lambda(x'-x)\big) - f(t,x)\Big).
\end{align*}
For any fixed $(t',x')$ and sufficiently small $\lambda$, the assumption that $f-\phi$ has a local minimum at $(t,x)$ gives
\begin{align*}
    f\big(t+\lambda(t'-t),x+\lambda(x'-x)\big) - f(t,x)\geq \phi\big(t+\lambda(t'-t),x+\lambda(x'-x)\big) - \phi(t,x).
\end{align*}
Using the above two displays and setting $\lambda\to0$, we obtain
\begin{align}\label{eq:sup_sol_supp_0}
    f(t',x')-f(t,x)\geq r(t'-t) + \big(\nabla\phi(x,t)\big)\cdot(x'-x)
\end{align}
where
\begin{align}\label{eq:r_q}
    r=\partial_t\phi(x,t).
\end{align}

\smallskip

Before proceeding, we make a digression to convex analysis. Most of the definitions and results we need are given in Appendix~\ref{section:biconjugation}. For each fixed $t\geq 0$, it can be seen from \eqref{eq:f_Hopf_1} that $f(t,\cdot)$ is convex. Setting $t'=t$ in \eqref{eq:sup_sol_supp_0}, we have $\nabla\phi(x,t)\in\partial f(t,x)$ which stands for the subdifferential of $f(t,\cdot)$ at $x$. Its definition is given in \eqref{eq:def_subdiff}. Invoking Lemma~\ref{lemma:subdiff}, we can express
\begin{align}\label{eq:q=a+b}
    \nabla \phi(x,t) =a +b
\end{align}
where $b\in \nn(x)$, the outer normal cone at $x$, defined in \eqref{eq:def_normal_cone}; and $a $ belongs to the closed convex hull of limit points of the form $ \lim_{n\to\infty}\nabla f(t,x_n)$ where $\lim_{n\to\infty}x_n=x$ and $f(t,\cdot)$ is differentiable at each $x_n$. Since $f$ is nondecreasing and Lipschitz, we have
\begin{align}\label{eq:a_in_S}
    a\in\S^K_+,\qquad |a|\leq \|f\|_\mathrm{Lip}.
\end{align}
By the definition of $\nn(x)$ and Lemma~\ref{lemma:psd}, it can seen that $-b\in\S^K_+$. This along with \eqref{eq:q=a+b} implies
\begin{align}\label{eq:a>q}
    a \in \nabla \phi(t,x)+\S^K_+.
\end{align}
By Lemma~\ref{lemma:subdiff}, the definition of $a$ and an easy observation that $0\in\nn(x)$, we can deduce that $a\in\partial f(t,x)$, which due to the definition of subdifferential in \eqref{eq:def_subdiff} further implies
\begin{align*}
    f(t,x')-f(t,x)\geq a \cdot(x'-x),\quad \forall x'\in\S^K_+.
\end{align*}
Set $x'=x$ in \eqref{eq:sup_sol_supp_0} and use the above display to get
\begin{align}\label{eq:sup_sol_supp}
    f(t',x')-f(t,x)\geq r(t'-t) + a\cdot(x'-x),\quad \forall t'\geq 0,\, x'\in\S^K_+.
\end{align}

\smallskip

Now, we return to the proof. For each $s\geq 0$, we define
\begin{align*}
    \eta_s(x')=f(t,x)-rs + a\cdot(x'-x),\quad\forall x'\in\S^K_+.
\end{align*}
Setting $t'=t-s$ in \eqref{eq:sup_sol_supp}, for $s\in[0,t]$, we have
\begin{align*}
    f(t-s,x')\geq \eta_s(x'),\quad\forall x'\in\S^K_+.
\end{align*}
Applying the order-reversing property of the Fenchel transform twice, we obtain from the above display that
\begin{align*}
    \big(f^*(t-s,\cdot)-s\H\big)^*\geq  \big(\eta_s^*-s\H\big)^*.
\end{align*}
Due to the semigroup property \eqref{eq:semigroup}, this yields
\begin{align*}
    f(t,\cdot)\geq \big(\eta_s^*-s\H\big)^*,\quad \forall s\in[ 0,t].
\end{align*}
By \eqref{eq:a_in_S} and the definition of the Fenchel transform in \eqref{eq:def_u*}, the above yields
\begin{align*}
    f(t,x)\geq a\cdot x-\eta_s^*(a)+s\H(a).
\end{align*}
On the other hand, using the definition of $\eta_s$, we can compute
\begin{align*}
    \eta_s^*(a)= -f(t,x)+rs+a\cdot x.
\end{align*}
Combine the above two displays with \eqref{eq:r_q} and that these hold for all $s\in [0,t]$ to see
\begin{align*}
    \big(\partial_t\phi-\H(a)\big)(t,x)\geq 0.
\end{align*}

\smallskip

If $x \in \S^K_{++}$, then \eqref{eq:a_in_S}, \eqref{eq:a>q} and the nondecreasingness of $\H$ in Lemma~\ref{lemma:DH_psd} imply $\H(a)\geq \H(\nabla\phi(t,x))$. If $x\in \S^K_+\setminus\S^K_{++}$, then those same ingredients yield $\H(a)\geq \inf\H(q)$ where the infimum is described in \eqref{item:vs_2} in Definition~\ref{def:vs}. These along with the above display verifies that $f$ is a supersolution.

\section{Convergence to the viscosity solution}\label{section:cvg_vis_sol}

The goal of this section is to prove Theorem~\ref{thm:lower}. We first state the main result of this section and deduce Theorem~\ref{thm:lower} from it.

\begin{Prop}\label{Prop:limit_sub_sol}
Under the assumptions in Theorem~\ref{thm:lower}, suppose that a subsequence of $(\bF_N)_{N\in\N}$ converges locally uniformly to some function $f:\R_+\times \S^K_+\to \R$. Then, $f$ is a viscosity subsolution to \eqref{eq:HJ_eqn} with $f(0,\cdot)=\psi$. If $\H$ is convex, then $f$ is also a supersolution and thus the unique viscosity solution to \eqref{eq:HJ_eqn}.
\end{Prop}

\begin{Rem}\label{Rem:strong_subsol}
In fact, if $f$ is a subsequential limit of $(\bF_N)_{N\in\N}$, then $f$ satisfies the following: if $f-\phi$ achieves a local maximum at $(t,x)\in(0,\infty)\times \S^K_{++}$ for a smooth function $\phi$, then it holds that
\begin{align*}
    \big(\partial_t \phi - \H(\nabla\phi)\big)(t,x) =0,
\end{align*}
which is stronger than Definition~\ref{def:vs}~\eqref{item:vs_1}.
\end{Rem}

\begin{proof}[Proof of Theorem~\ref{thm:lower}]
By \eqref{eq:2nd_der_bF}, \eqref{eq:1st_der_bF}, \eqref{eq:1st_der_bF_lower} and the assumption that $\bF_N(0,\cdot)$ converges to $\psi$ pointwise, we have that $\psi$ is convex, Lipschitz and nondecreasing. Hence,  Proposition~\ref{Prop:hopf_viscosity} implies that there is a Lipschitz viscosity solution $f$ to the Hamilton--Jacobi equation \eqref{eq:HJ_eqn} with $f(0,\cdot)=\psi$.  Proposition~\ref{prop:comp_principle} ensures the uniqueness.

\smallskip

Since $\bF_N(0,0)=0$ for all $N$ and $(\bF_N)_{N\geq 1}$ is Lipschitz uniformly in $N$ due to \eqref{eq:1st_der_bF}, the Arzel\`a--Ascoli theorem guarantees that any subsequence of $(\bF_N)_{N\geq 1}$ has a further subsequence that converges in the local uniform topology to some function $g$. In addition, we can see that $g$ is Lipschitz. The assumption on $\psi$ in Theorem~\ref{thm:lower} ensures that $g(0,\cdot)=\psi$. Proposition~\ref{Prop:limit_sub_sol} implies that $g$ is a viscosity subsolution to \eqref{eq:HJ_eqn}. The upper bound in Theorem~\ref{thm:lower} then follows from Proposition~\ref{prop:comp_principle}. When $\H$ is convex, using similar arguments, we can obtain an lower bound.
\end{proof}

We prove the subsolution part of Proposition~\ref{Prop:limit_sub_sol} and Remark~\ref{Rem:strong_subsol} in Section~\ref{section:subsol} and the supersolution part of Proposition~\ref{Prop:limit_sub_sol} in Section~\ref{section:supsol}.

\subsection{The limit is a subsolution}\label{section:subsol}

To lighten the notation, we assume $\bF_N$ converges to $f$ locally uniformly. We want to show $f$ is subsolution to \eqref{eq:HJ_eqn}.

First, we consider the case where $f-\phi$ has a local maximum at $(t,h)$ with $t>0$ and $h\in \S^K_+\setminus\S^K_{++}$. Then, there is a sequence $\big((t_N, h_N)\big)_{N\in\N}$ in $ (0,\infty)\times \S^K_+$ such that $(t_N,h_N)$ converges to $(t,h)$ and $\bF_N-\phi$ has a local maximum at $(t_N,h_N)$. Note that $a+h_N\in \S^K_+$ for all $a\in \S^K_+$. So, we can differentiate $\bF_N-\phi$ along any direction $a\in \S^K_+$ to see
\begin{align*}
    a\cdot \nabla\big(\bF_N-\phi\big)(t_N,h_N) \leq 0,\quad \forall a \in \S^K_+.
\end{align*}
In view of \eqref{eq:1st_der_bF_lower}, this implies
\begin{align*}
    a\cdot \nabla \phi(t_N,h_N) \geq 0,\quad \forall a \in \S^K_+.
\end{align*}
Setting $N\to\infty$, by Lemma~\ref{lemma:psd}, we have $\nabla \phi(t,h)\in \S^K_+$, verifying the boundary condition for subsolutions.

\smallskip

Now, we study the case when $f-\phi$ achieves a local maximum at $(t,h)$ with $t>0$ and $h\in\S^K_{++}$. In the following, the constant $C$ is allowed to depend on $t$, $h$, $f$, $\phi$. We set
\begin{gather}
    M=(t\vee|h|)+1, \label{eq:M=(tv|h|)+1}\\
    \gamma = K(K+1)/2, \label{eq:gamma_visc}\\
    \delta_N =\|\bF_N -f \|^\frac{1}{4}_{L^\infty([0,M]\times \S^K_{+,M})} + \cK_{M,N}^\frac{1}{2} , \label{eq:delta_N}
\end{gather}
where $\cK_{M,N}$ is defined in \eqref{eq:def_K} and $\S^K_{+,M}$ is given in \eqref{eq:def_SM}. By the convergence of $\bF_N$ to $f$ and the assumption \eqref{eq:cK_decay_beta}, we have $\lim_{N\to\infty}\delta_N =0$.
Let us introduce 
\begin{align}\label{eq:tphi}
    \tphi(t', h') =\phi (t',h') + |t'-t|^2 + |h'-h|^2.
\end{align}
It is immediate that $f-\tphi$ has a local maximum at $(t,h)$.
Due to \eqref{eq:delta_N}, for all $(t',h')\in [0,M]\times \S^K_{+,M}$, we have
\begin{align*}
    \big(\bF_N - \tphi\big)(t',h') &\leq \big(f - \phi\big) (t',h') - |t'-t|^2-|h'-h|^2+\delta_N^4,\\
    \big(\bF_N-\tphi\big)(t,h) &\geq \big(f-\phi\big)(t,h)-\delta_N^4.
\end{align*}

Since $\bF_N$ converges locally uniformly to $f$, for $N$ large, there is a sequence of $(t_N,h_N)$ in $(0,\infty)\times \S^K_{++}$, at which $\bF_N - \tphi$ attains a local maximum, and which converges to $(t,h)$.
From the above display and the fact that $f-\phi$ attains a local maximum at $(t,h)$, we can deduce that
\begin{align}\label{eq:h-h_N_est}
    |t_N-t|^2 + |h_N-h|^2 \leq 2 \delta_N^4.
\end{align}
By the definition of $(t_N, h_N)$, we also have 
\begin{align}\label{eq:F_N-tphi_1st_der}
    \partial_t \big(\bF_N - \tphi\big)(t_N,h_N) =0, \quad \nabla\big(\bF_N -\tphi\big)(t_N,h_N) =0.
\end{align}

\bigskip

We want to apply Proposition~\ref{Prop:approx_HJ}. However the concentration estimate we have is for $F_N - \bF_N$ not for $\nabla(F_N-\bF_N)$. Therefore, we need to do a local average by introducing
\begin{align}
    \label{eq:D_N} D_N &  = \S^K_{+,\delta_N},\\
    \label{eq:G_N} G_N(t',h') &= |D_N|^{-1}\int_{h'+D_N}\bF_N(t',h'')\d h''.  
\end{align}
It is clear that $G_N$ converges locally uniformly to $f$. Hence, there is $(t'_N, h'_N)\in(0,\infty)\times\S^K_{++}$ converging to $(t,h)$ such that $G_N -\tphi$ has a local maximum at $(t'_N,h'_N)$. Consequently, we have
\begin{gather}
    \label{eq:G_N_1st_der} \partial_t \big(G_N-\tphi\big)(t'_N,h'_N) = 0, \quad  \nabla \big(G_N-\tphi\big)(t'_N,h'_N) = 0,\\
    \label{eq:G_N_2nd_der} a\cdot \nabla\Big( a \cdot \nabla \big(G_N - \tphi\big)\Big)(t'_N,h'_N) \leq 0,\quad \forall a \in \S^K.
\end{gather}
Repeating the argument in the derivation of \eqref{eq:h-h_N_est} yields
\begin{align}\label{eq:h-h_N'_est}
    |t'_N -t|^2 + |h'_N-h|^2 \leq 2 \delta_N^4.
\end{align}

We need the following estimates:
\begin{gather}
    \label{eq:grad_F_N_concent} \int_{h'_N + D_N}\E\Big| \nabla F_N - \nabla \bF_N\Big|^2 (t'_N, h') \d h' \leq C\delta_N^{\gamma+1},\\
    \label{eq:bF_N-G_N_grad}   \int_{h'_N +D_N}\Big|\nabla \bF_N (t'_N, h') - \nabla G_N(t'_N, h'_N)\Big|^2 \d h' \leq C \delta^{\gamma+1}_N.
\end{gather}
From the definition of $\H$ in \eqref{eq:H}, we can see that $|\H(a)-\H(b)|\leq C|a-b|(|a|\vee|b|)^{p-1}$ for all $a,b\in \S^K_+$. By this, Jensen's inequality and \eqref{eq:bF_N-G_N_grad}, we have
\begin{align}\label{eq:compare_H(nabla_F)_H(nabla_G)}
\begin{split}
    &\bigg|\  |D_N|^{-1}\int_{\hN'+D_N}\H\big(\nabla \bF_N(\tN',h')\big)\d h'-\H\big(\nabla G_N(\tN',\hN')\big)\bigg|\\
    &\quad\leq C\bigg(|D_N|^{-1} \int_{h'_N +D_N}\Big|\nabla \bF_N (t'_N, h') - \nabla G_N(t'_N, h'_N)\Big|^2 \d h'\bigg)^\frac{1}{2}\leq C\delta_N^\frac{1}{2}.
\end{split}
\end{align}
Here, we used the following fact due to \eqref{eq:gamma_visc} and \eqref{eq:D_N}
\begin{align*}
    |D_N|=C\delta_N^\gamma.
\end{align*}
Recall the definition of $\kappa$ in \eqref{eq:def_kappa(h)}. Due to $h\in\S^K_{++}$, \eqref{eq:D_N} and \eqref{eq:h-h_N'_est}, we know that $\kappa(h')\leq C$ for all $h'\in h'_N+D_N$ and $N$ large. Take average of $(\partial_t\bFN-\H(\nabla\bFN))(\tN',h')$ over $\hN'+D_N$, and use Proposition~\ref{Prop:approx_HJ} and \eqref{eq:compare_H(nabla_F)_H(nabla_G)} to see
\begin{align*}
    &\Big|\partial_tG_N-\H\big(\nabla G_N\big)\Big|(\tN',\hN')\leq C\delta_N^\frac{1}{2}\\
    &+C\Bigg(N^{-\frac{1}{4}}\fint_{\hN'+D_N}\big(\Delta\bFN(t'_N, h') +1\big)^\frac{1}{4} \d h'+ \fint_{\hN'+D_N}\E\Big| \nabla F_N - \nabla \bF_N\Big|^2 (t'_N, h') \d h' \Bigg)^\frac{1}{2}
\end{align*}
where $\fint_{\hN'+D_N} = |D_N|^{-1}\int_{\hN'+D_N}$. By Jensen's inequality, \eqref{eq:G_N} and \eqref{eq:G_N_2nd_der}, we have
\begin{align*}
    \fint_{\hN'+D_N}\big(\Delta\bFN(\tN',\cdot) +1\big)^\frac{1}{4} \leq \Big(\Delta G_N(\tN',\hN')+1\Big)^\frac{1}{4} \leq C.
\end{align*}
The above two displays along with \eqref{eq:grad_F_N_concent} give
\begin{align*}
    \Big|\partial_tG_N-\H\big(\nabla G_N\big)\Big|(\tN',\hN')\leq C \Big( \delta_N^\frac{1}{2} + N^{-\frac{1}{8}}\Big).
\end{align*}
Using \eqref{eq:h-h_N'_est} and \eqref{eq:G_N_1st_der}, and sending $N$ to $\infty$, we obtain
\begin{align*}
    \partial_t\tphi-\H\big(\nabla \tphi\big)(t,h) = 0.
\end{align*}
Due to \eqref{eq:tphi}, the derivatives of $\tphi$ coincide with those of $\phi$ at $(t,h)$. This finishes the core of the verification of that $f$ is a subsolution and the claim in Remark~\ref{Rem:strong_subsol}.

\bigskip

To complete the proof, we derive \eqref{eq:grad_F_N_concent} and \eqref{eq:bF_N-G_N_grad}.

\begin{proof}[Proof of \eqref{eq:grad_F_N_concent}]
For any smooth $g: \S^K_+ \to \R$ and any $D\subset \S^K_+$ with Lipschitz boundary, integration by parts gives
\begin{align}\label{eq:ibp}
    \int_D |\nabla g|^2 = \int_{\partial D}g\nabla g \cdot \mathbf{n} - \int_D g\Delta g,
\end{align}
where $\mathbf{n}$ is the outer normal on $\partial D$. To lighten our notation, the time variable is always evaluated at $\tN'$ in this proof. Apply \eqref{eq:ibp} to get
\begin{align}
    \int_{h'_N + D_N}\Big| &\nabla F_N - \nabla \bF_N\Big|^2 \leq \|F_N-\bFN\|_{L^\infty(\hN'+D_N)}\nonumber \\
    &\times\bigg(\int_{\partial(h'_N + D_N)}\big|\nabla F_N - \nabla \bF_N\big| + \int_{h'_N + D_N}\big|\Delta F_N - \Delta \bF_N\big|\bigg).\label{eq:concent_grad_est}
\end{align}
By $\lim_{N\to\infty}h'_N=h$ (due to \eqref{eq:h-h_N'_est}), $h\in\S^K_{++}$ and \eqref{eq:D_N}, we have $|h'^{-1}|\leq C$ for all $h'\in h'_N+D_N$ for large $N$,
Using this, \eqref{eq:2nd_der_bF} and \eqref{eq:2nd_der_F}, we get, for all $h'\in h'_N + D_N$,
\begin{align*}
    \big|\Delta F_N - \Delta \bF_N\big| \leq \Delta F_N + \Delta \bF_N + CN^{-\frac{1}{2}}|Z|.
\end{align*}
Applying this and integration by parts to obtain
\begin{align*}
    \int_{h'_N + D_N}\big|\Delta F_N - \Delta \bF_N\big| \leq C\delta_N^\gamma N^{-\frac{1}{2}}|Z|+ \int_{\partial(h'_N + D_N)}\big|\nabla F_N|+|\nabla\bF_N|.
\end{align*}
Then, using this display, \eqref{eq:1st_der_bF} and \eqref{eq:1st_der_F}, we can bound the two integrals in \eqref{eq:concent_grad_est} by $C\delta_N^{\gamma-1}(1 + N^{-\frac{1}{2}}|Z|)$. As a result, by taking expectations and invoking the Cauchy--Schwarz inequality in \eqref{eq:concent_grad_est}, we obtain
\begin{align}\label{eq:to_compare}
    \E\int_{h'_N + D_N}\Big| &\nabla F_N - \nabla \bF_N\Big|^2 \leq C\delta_N^{\gamma-1}\Big( \E \|F_N-\bFN\|_{L^\infty(\hN'+D_N)}^2\Big)^\frac{1}{2}.
\end{align}
Recall that the time variable is evaluated at $t'_N$. By \eqref{eq:M=(tv|h|)+1}, \eqref{eq:h-h_N'_est} and \eqref{eq:D_N}, we have $\{t'_N\}\times(h'_N+D_N)\subset [0,M]\times \S^K_{+,M}$ for large $N$. Hence, the desired result \eqref{eq:grad_F_N_concent} follows from \eqref{eq:delta_N} and the definition \eqref{eq:def_K}.

\end{proof}

\bigskip

\begin{proof}[Proof of \eqref{eq:bF_N-G_N_grad}]
To prepare, we start by showing that, for $h'$ satisfying $|h'-h_N|\leq C^{-1}$,
\begin{align}
    \label{eq:taylor_bF}\Big|\bF_N(t_N, h') - \bF_N(t_N, h_N) - (h'-h_N)\cdot \nabla\bF_N(t_N,h_N)\Big|\leq C|h'-h_N|^2.
\end{align}
By Taylor expansion, we have
\begin{align}\label{eq:bF_N_expansion}
\begin{split}
    \bF_N(t_N,h') &- \bF_N(t_N,h_N) = (h' - \hN)\cdot \nabla \bFN(\tN,\hN)\\
    &\quad +\int_0^1 (1-r) \cD^2_{h'-\hN}\bFN(\tN,\hN+(h'-\hN)r)\d r
\end{split}
\end{align}
where we write 
\begin{align*}
    \cD^2_a \bF_N = a\cdot \nabla\big(a \cdot \nabla \bFN  \big),\quad \forall a \in \S^K.
\end{align*}
A similar equation also holds with $\bF_N$ replaced by $\tphi$. Take the difference of these two equations and use \eqref{eq:F_N-tphi_1st_der} and the fact that $\bFN-\tphi$ has a local maximum at $(\tN,\hN)$ to see
\begin{align*}
    \int_0^1 (1-r) \cD^2_{h'-\hN}\bFN(\tN,\hN+(h'-\hN)r)\d r\\
    \leq \int_0^1 (1-r) \cD^2_{h'-\hN}\tphi(\tN,\hN+(h'-\hN)r)\d r.
\end{align*}
Since $\tphi$ has locally bounded derivatives, by the above display and \eqref{eq:2nd_der_bF}, there is $C$ such that the following holds for all $h'$ with $|h'-\hN|\leq C^{-1}$
\begin{align*}
    \Bigg|\int_0^1 (1-r) \cD^2_{h'-\hN}\bFN(\tN,\hN+(h'-\hN)r)\d r\Bigg|\leq C|h'-\hN|^2.
\end{align*}
Inserting this into \eqref{eq:bF_N_expansion} gives \eqref{eq:taylor_bF}.

\bigskip

Now, we are ready to prove \eqref{eq:bF_N-G_N_grad}. Let us set
\begin{align*}
    g_N(h') = \bFN(\tN',h') - \bFN(\tN',\hN')-(h'-\hN')\cdot \nabla G_N(\tN',\hN').
\end{align*}
Note that, to probe \eqref{eq:bF_N-G_N_grad}, it is sufficient to estimate $\int_{h'_N+D_N}|\nabla g_N|^2$. Using \eqref{eq:1st_der_bF} and \eqref{eq:G_N}, we can see
\begin{align}\label{eq:grad_G_N}
    |\nabla G_N(t',h')|\leq C,\quad \forall t',\ h'.
\end{align}
\begin{align}\label{eq:grad_g_N_est}
    |\nabla g_N(h')|\leq C,\quad \forall h' \in \hN'+D_N.
\end{align}
Apply \eqref{eq:ibp} to $g_N$ to obtain
\begin{align*}
    \int_{\hN'+D_N}|\nabla g_N|^2\leq \|g_N\|_{L^\infty(\hN'+D_N)}\bigg(\int_{\partial(\hN'+D_N)}\big| \nabla g_N\big| +\int_{\hN'+D_N}\big|\Delta g_N\big|\bigg).
\end{align*}
By \eqref{eq:grad_g_N_est}, the first integral on the left is bounded by $C\delta_N^{\gamma-1}$. Since $\Delta g_N = \Delta \bFN(\tN',\cdot)$, by \eqref{eq:2nd_der_bF}, we can see $|\Delta g_N| = \Delta g_N$. Integrating by parts and applying \eqref{eq:grad_g_N_est} again, we deduce that the last integral in the above display is also bounded by $C\delta_N^{\gamma-1}$. Hence, we arrive at
\begin{align}\label{eq:int_|grad_g_N|^2}
    \int_{\hN'+D_N}|\nabla g_N|^2\leq C\delta_N^{\gamma-1} \|g_N\|_{L^\infty(\hN'+D_N)}.
\end{align}

\smallskip

It remains to estimate $\|g_N\|_{L^\infty(\hN'+D_N)}$. We want to compare $g_N$ with 
\begin{align*}
    \bF_N(t_N,h')-\bF_N(t_N,h_N)-(h'-h_N)\cdot\nabla G_N(t_N,h_N).
\end{align*}
To start, using \eqref{eq:G_N}, we can compute, for all $a, t', h'$,
\begin{align}
     a\cdot \nabla G_N(t',h')& = |D_N|^{-1}\int_{D_N} a\cdot \nabla\bFN(t', h'+h'') \d h''\nonumber\\
    \label{eq:grad_G_N_surface} & = |D_N|^{-1}\int_{\partial D_N} \bFN(t', h'+h'')a\cdot \mathbf{n}  \cS(\d h'')
\end{align}
where in the last equality we used integration by parts and $\cS$ denotes the surface measure on $\partial D_N$. 
Now, we estimate
\begin{align}
    &\Big|(h'-h_N)\cdot\nabla G_N(t_N,h_N)-(h'-h'_N)\cdot\nabla G_N(t'_N,h'_N)\Big|\label{eq:inter_compare_g_N}\\ 
    &\leq |h_N-h'_N|\big|\nabla G_N(t_N,h_N)\big|+ \Big|(h'-h'_N)\cdot\Big(\nabla G_N(t_N,h_N)-\nabla G_N(t'_N,h'_N)\Big)\Big|.\nonumber
\end{align}
The first term after the inequality sign is bounded by $|h_N-h'_N|$ due to \eqref{eq:grad_G_N}. Using \eqref{eq:1st_der_bF} and \eqref{eq:grad_G_N_surface}, we can bound the second term by
\begin{align*}
    |D_N|^{-1}\int_{\partial D_N}\Big(|t_N-t'_N|+|h_N-h'_N|\Big)|h'-h'_N|\leq C|t_N-t'_N|+C|h_N-h'_N|,
\end{align*}
for all $h'\in h'_N+D_N$.  Hence, we conclude that \eqref{eq:inter_compare_g_N} is bounded by the right hand of the above display with a larger constant. This along with \eqref{eq:1st_der_bF} implies that
\begin{align*}
    \|g_N &\|_{L^\infty(\hN'+D_N)} \leq C|\tN-\tN'|+C|\hN-\hN'|\\
    &+ \sup_{h'\in \hN'+D_N} \Big|\bFN(\tN,h') - \bFN(\tN,\hN)-(h'-\hN)\cdot \nabla G_N(\tN,\hN)\Big|.
\end{align*}
By \eqref{eq:taylor_bF} and the definition of $D_N$ in \eqref{eq:D_N}, the supremum above can be bounded by
\begin{align*}
    C\big(\delta_N &+ |\hN-\hN'|\big)^2\\
    &+\sup_{h'\in \hN'+D_N}\big|(h'-h_N)\cdot \nabla\bF_N(t_N,h_N)-  (h'-\hN)\cdot \nabla G_N(\tN,\hN)\big|.
\end{align*}
We claim that
\begin{align}\label{eq:grad_bF_N-G_N}
    \sup_{h'\in \hN'+D_N}\big|(h'-h_N)\cdot \nabla\bF_N(t_N,h_N)-  (h'-\hN)\cdot \nabla G_N(\tN,\hN)\big|\leq C\delta^2_N.
\end{align}
This along with \eqref{eq:h-h_N_est} and \eqref{eq:h-h_N'_est} implies that $ \|g_N \|_{L^\infty(\hN'+D_N)} \leq C\delta^2_N$. Plug this into \eqref{eq:int_|grad_g_N|^2}, and we obtain \eqref{eq:bF_N-G_N_grad}.

\medskip

To complete the proof, we verify the claim \eqref{eq:grad_bF_N-G_N}. Using integration by parts, we can see
\begin{align*}
    (h'-h_N)\cdot \nabla\bF_N(t_N,h_N) = |D_N|^{-1}\int_{\partial D_N}\Big(h''\cdot \nabla \bFN(\tN, \hN)\Big)(h'-\hN)\cdot \mathbf{n}\cS(\d h'').
\end{align*}
Using the formula \eqref{eq:grad_G_N_surface} and $\int_{\partial D_N}c\cdot \nn =0 $ for any constant vector $c$, we can also get
\begin{align*}
    (h'-\hN)\cdot &\nabla G_N(\tN,\hN) = \\
    &|D_N|^{-1}\int_{\partial D_N}\Big( \bFN(\tN, \hN+h'')- \bFN(\tN, \hN)\Big)(h'-\hN)\cdot \mathbf{n}  \cS(\d h'').
\end{align*}
Taking the difference of the above two equations and using \eqref{eq:taylor_bF}, we can see the left hand side of \eqref{eq:grad_bF_N-G_N} is bounded by 
\begin{align*}
    C \sup_{h'\in \hN'+D_N}\delta_N|h'-h_N|\leq C\delta_N\big( \delta_N+|\hN-\hN'|\big).
\end{align*}
Now, \eqref{eq:grad_bF_N-G_N} follows from \eqref{eq:h-h_N_est} and \eqref{eq:h-h_N'_est}.

\end{proof}

\subsection{The limit is a supersolution when \texorpdfstring{$\H$}{H} is convex}\label{section:supsol}
Under the additional assumption that $\H$ is convex, we show that any subsequential limit of $\bF_N$ is a supersolution. For simplicity of notation, we again assume the entire sequence $(\bF_N)_{N\in\N}$ converges locally uniformly to $f$. Suppose $f-\phi$ achieves a local minimum at $(t,h)\in(0,\infty)\times \S^K_+$. Recall $M$ from \eqref{eq:M=(tv|h|)+1}. Let us redefine
\begin{gather}
    \delta_N = \max\{N^{-\frac{1}{6}},\cK_{M,N}^{\frac{2}{5}}\}, \label{eq:delta_N_2}\\
    D_N = \delta_N I + \S^K_{+,\delta_N},\nonumber\\
    G_N(t',h') = |D_N|^{-1}\int_{h'+D_N}\bFN(t',h'')\d h'',\quad \forall (t',h')\in\R_+\times \S^K_+.\label{eq:G_N_supsol}
\end{gather}
Note that in the definition of $G_N$, the integration is over a region away from $h'$ to avoid the singularity present in the right hand side of the estimate in Proposition~\ref{Prop:approx_HJ}. It is clear that $G_N$ converges locally uniformly to $f$. Then, there is a sequence $(t_N,h_N)\in(0,\infty)\times \S^K_+$ such that $\lim_{N\to\infty}(t_N,h_N)=(t,h)$ and $G_N-\phi$ has a local minimum at $(t_N,h_N)$. Since $\H$ is convex, we integrate both sides of the inequality in Proposition~\ref{Prop:approx_HJ} and use Jensen's inequality to see
\begin{align}
    \Big(\partial_tG_N-\H\big(\nabla G_N\big)\Big)(\tN,\hN) \geq \fint_{h_N+D_N} \Big(\partial_t \bF_N - \H(\nabla\bF_N)\Big)(t_N,h')\d h' \nonumber\\ \geq -C\Bigg(\fint_{\hN+D_N}\frac{|h'^{-1}|}{N^{\frac{1}{4}}}\big(\Delta\bFN +|h'^{-1}|\big)^\frac{1}{4} \d h'  + \fint_{\hN+D_N}\E\Big| \nabla F_N - \nabla \bF_N\Big|^2   \Bigg)^\frac{1}{2}\label{eq:supsol_G_N_lower}
\end{align}
where $\fint_{\hN+D_N} = |D_N|^{-1}\int_{\hN+D_N}$ and the time variable is evaluated at $t_N$ in \eqref{eq:supsol_G_N_lower}.

\smallskip

Let us estimate the integrals in \eqref{eq:supsol_G_N_lower}. 
The definition of $D_N$ implies that
\begin{align}\label{eq:h_inverse_norm}
    |h'^{-1}|\leq C\delta_N^{-1},\quad \forall h'\in h_N+ D_N.
\end{align}
Integrate by parts and use \eqref{eq:1st_der_bF} to see
\begin{align*}
    \Delta G_N(t_N,h_N) =\fint_{\hN+D_N}\Delta\bFN(\tN,h')\d h' \leq |D_N|^{-1}\int_{\partial(\hN+D_N)}\big|\nabla\bFN(\tN,\cdot)\big|\leq C\delta_N^{-1}.
\end{align*}
The above two displays together with Jensen's inequality and \eqref{eq:delta_N_2} implies that
\begin{align}
    \fint_{\hN+D_N}&  N^{-\frac{1}{4}}|h'^{-1}|\big(\Delta\bFN(t_N, h') +|h'^{-1}|\big)^\frac{1}{4} \d h'\nonumber\\
    &\leq  CN^{-\frac{1}{4}}\delta_N^{-1}\bigg(\Delta G_N(t_N,h_N) + \delta_N^{-1}\bigg)^\frac{1}{4}\leq C\delta_N^\frac{1}{4}.\label{eq:supsol_est_1}
\end{align}
To estimate the last integral in \eqref{eq:supsol_G_N_lower}, we use the same argument in the proof of \eqref{eq:grad_F_N_concent}. The only difference is that since now it is possible that $h\in \S^K_+\setminus\S^K_{++}$, the singularity in the estimate \eqref{eq:1st_der_F} takes effect. Due to \eqref{eq:h_inverse_norm}, compared with \eqref{eq:to_compare}, there is an additional $\delta_N^{-\frac{1}{2}}$. For $N$ large, we have
\begin{align}
    \E\int_{h_N + D_N}\Big| \nabla F_N - \nabla \bF_N\Big|^2 & \leq C\delta_N^{\gamma-\frac{3}{2}}\Big( \E \|F_N-\bFN\|_{L^\infty(\hN+D_N)}^2\Big)^\frac{1}{2}\nonumber\\
    &\leq C\delta_N^{\gamma-\frac{3}{2}}\cK_{M,N}\leq C\delta_N^{\gamma+1},\label{eq:supsol_est_2}
\end{align}
where we used \eqref{eq:def_K} and \eqref{eq:M=(tv|h|)+1} in the penultimate inequality, and \eqref{eq:delta_N_2} in the last inequality.
Inserting \eqref{eq:supsol_est_1} and \eqref{eq:supsol_est_2} into \eqref{eq:supsol_G_N_lower}, we obtain
\begin{align}\label{eq:sup_G_N_lower}
     \Big(\partial_tG_N-\H\big(\nabla G_N\big)\Big)(\tN,\hN)\geq -C\delta^\frac{1}{8}_N.
\end{align}
First suppose that there are infinitely many $(t_N,h_N)$ with $h_N\in \S^K_{++}$. Since first derivatives of $G_N$ coincides with $\phi$ at those $(t_N,h_N)$, by taking $N\to\infty$ and using the smoothness of $\phi$, we obtain from \eqref{eq:sup_G_N_lower} that
\begin{align}\label{eq:check_supersol}
    \Big(\partial_t\phi-\H\big(\nabla \phi\big)\Big)(t,h)\geq 0.
\end{align}

\medskip

If there are infinitely many $(t_N,h_N)$ with $h_N\in \S^K_+\setminus\S^K_{++}$, then we must have $h \in \S^K_+\setminus\S^K_{++}$. Due to $t\in(0,\infty)$ and $\lim_{N\to\infty}t_N =t$, for large $N$, we have $t_N\in(0,\infty)$. Since $G_N-\phi $ has a local minimum at $(t_N,h_N)$, we have
\begin{gather}
     \big(\partial_tG_N - \partial_t\phi\big)(\tN,\hN)=0,\label{eq:sup_boundary_dt}\\
     \big(\nabla G_N - \nabla\phi\big)(\tN,\hN)\in\S^K_+.\label{eq:sup_boundary_grad}
\end{gather}
We also used Lemma~\ref{lemma:psd} in deriving \eqref{eq:sup_boundary_grad}.
By the definition of $G_N$ in \eqref{eq:G_N_supsol}, the nondecreasingness of $\bF_N$ in \eqref{eq:1st_der_bF_lower}, and the uniform Lipschitzness of $\bF_N$ in \eqref{eq:1st_der_bF}, we have, for all $N\in\N$,
\begin{align}\label{eq:grad_G_N_psd}
    \nabla G_N\in \S^K_+,\qquad |\nabla G_N|\leq \|\bF_N\|_\mathrm{Lip}\leq C,
\end{align}
where the last constant $C$ is absolute.
In addition, due to \eqref{eq:2nd_der_bF}, $ G_N$ is convex in the second variable, which yields
\begin{align*}
	y\cdot \nabla G_N(t_N,h_N)\leq G_N(t_N,h_N+y)-G_N(t_N,h_N),\quad \forall y \in \S^K_+.
\end{align*}
Let $a$ be any subsequential limit of $\big(\nabla G_N(t_N,h_N)\big)_{N\in\N}$. Replace $y$ by $y_N= \nabla G_N(t_N,h_N)$ in the above display and use $\lim_{N\to\infty}(t_N,h_N)=(t,h)$ and the local uniform convergence of $G_N$ towards $f$ to see
\begin{align*}
	|a|^2 \leq f(t,h+a) - f(t,h).
\end{align*}
The Lipschitzness of $f$ implies 
\begin{align}\label{eq:|a|<|f|}
	|a|\leq \|f\|_\mathrm{Lip}.
\end{align}
We extract a subsequence from $\big(\nabla G_N(t_N,h_N)\big)_{N\in\N}$, along which 
\begin{align*}
    \liminf_{N\to\infty}\H\big(\nabla G_N(t_N,h_N)\big)
\end{align*}
is achieved. Denote by $a$ the further subsequential limit of this minimizing sequence. By this and the continuity of $\H$, we obtain
\begin{align}\label{eq:H(a)}
		\liminf_{N\to\infty}\H\big(\nabla G_N(t_N,h_N)\big) = \H(a).
\end{align}
Due to \eqref{eq:sup_boundary_grad}, \eqref{eq:grad_G_N_psd} and $\lim_{N\to\infty}(t_N,h_N)=(t,h)$, we also have
\begin{align}\label{eq:a_psd}
	a-\nabla\phi(t,h) \in \S^K_+,\qquad a\in\S^K_+.
\end{align}
Recall the quantity $\inf \H(q)$ for the boundary condition in \eqref{item:vs_2} of Definition~\ref{def:vs}. By \eqref{eq:|a|<|f|} and \eqref{eq:a_psd}, we have
\begin{align*}
	\H(a)\geq \inf \H(q).
\end{align*}
Use this, \eqref{eq:sup_boundary_dt}, \eqref{eq:H(a)} and \eqref{eq:sup_G_N_lower} to get
\begin{align*}
	\Big(\partial_t\phi-\inf \H(q)\Big)(t,h)\geq \lim_{N\to\infty} \partial_t G_N(t_N,h_N) - \H(a)\\
	= \lim_{N\to\infty} \partial_t G_N(t_N,h_N) - \liminf_{N\to\infty}\H\big(\nabla G_N(t_N,h_N)\big) \\
	\geq \limsup_{N\to\infty}\Big(\partial_t G_N-\H\big(\nabla G_N\big)\Big)(t_N,h_N)\geq 0.
\end{align*}
This along with \eqref{eq:check_supersol} completes our verification that $f$ is a supersolution.

\appendix

\section{Nonsymmetric matrix inference}\label{section:apdx_nonsym}

The goal of this appendix is to demonstrate a case where $\H$ is not convex, yet the assumptions on $\cA_\H$ in Theorem~\ref{thm:general_cvg_bF} are satisfied. Let $X_1$ and $X_2$ be two random vectors in $\R^N$. The task is to infer the nonsymmetric matrix $X_1X_2^\intercal$ from the noisy observation 
\begin{align}\label{eq:nonsym_infer}
    Y=\sqrt{\frac{2t}{N}}X_1X_2^\intercal + W\in \R^{N\times N}.
\end{align}
Let $X=\mathsf{diag}(X_1,X_2)\in \R^{2N\times 2}$. We can compute 
\begin{align*}
    X\otimes X = \mathsf{diag}\big(X_1\otimes X_1,\ X_1\otimes X_2,\ X_2\otimes X_1,\ X_2\otimes X_2 \big)\in \R^{4N^2\times 4}.
\end{align*}
Let $A = (0,1,0,0)\in\R^4$. Then note that the non-zero entries of $(X\otimes X) A$ are those from $X_1\otimes X_2$, which are exactly the entries of $X_1X_2^\intercal$. As observed in \cite{reeves2020information}, the model \eqref{eq:nonsym_infer} is equivalent to the model
\begin{align*}
    Y = \sqrt{\frac{2t}{N}}X^{\otimes 2}A+W\in \R^{4N^2\times 1},
\end{align*}
which is a special case of \eqref{eq:observ}.

\smallskip

By the formula of $\H$ in \eqref{eq:H}, we can compute $\H(q)=q_{11}q_{22}$ and thus $\cD\H(q) = \mathsf{diag}(q_{22},q_{11})$ for all $q\in\S^2_+$. Recall the set $\cA$ defined above \eqref{eq:A_H}. Then for smooth $\phi\in\cA$, using the basis \eqref{eq:orthonormal_basis}, we can obtain
\begin{align*}
    \nabla\cdot \cD\H(\nabla \phi) = 2 e^{11}\cdot \nabla(e^{22}\cdot \nabla \phi).
\end{align*}
Hence, formally, $\cA_\H$ consists of those $\phi\in\cA$ whose second order derivative as on the left of the above is nonnegative. By standard arguments involving test functions, we can see $\cA_\H$ is indeed convex. Then, we show $\bF_N(t,\cdot)\in\cA_\H$ for all $t$ and all $N$. In the proof of \eqref{eq:2nd_der_bF}, we used \cite[(3.27)]{mourrat2019hamilton} to compute $a\cdot\nabla(a\cdot \nabla \bF_N)$. A slight modification of \cite[(3.27)]{mourrat2019hamilton} gives
\begin{align*}
    &Na\cdot \nabla(b\cdot \nabla\bF_N)\\
    &= \E\big\la\big(a\cdot x^\intercal  x'\big)(b\cdot x^\intercal  x'\big)\big\ra-2\E\big\la\big(a\cdot x^\intercal  x'\big)\big(b\cdot x^\intercal x''\big)\big\ra+\E\big\la a\cdot x^\intercal x'\big\ra\la b\cdot x^\intercal x'\big\ra,
\end{align*}
for $a,b\in\S^2$. By the definition of $X$ in this model, under the Gibbs measure $\la\cdot \ra$, we can write $x = \mathsf{diag}(x_1,x_2)$ with $x_1,x_2\in\R^N$. Replace $a$ and $b$ by $e^{11}$ and $e^{22}$ respectively in the above display to see $Ne^{11}\cdot\nabla (e^{22}\cdot \bF_N)$ is given by
\begin{align*}
    \E\big\la\big(x_1\cdot x_1'\big)(x_2\cdot x_2'\big)\big\ra-2\E\big\la\big(x_1\cdot x_1'\big)\big(x_2\cdot x_2''\big)\big\ra+\E\big\la x_1\cdot x_1'\big\ra\la x_2\cdot x_2'\big\ra\\
    = \E\sum_{m,n=1}^N\bigg(\la x_{1,m}x_{2,n}\ra^2 - 2\la x_{1,m}x_{2,n}\ra \la x_{1,m}\ra \la x_{2,n} \ra+ \la x_{1,m}\ra \la x_{2,n}\ra^2\bigg)\geq 0.
\end{align*}
This shows that the assumptions on $\cA_\H$ in Theorem~\ref{thm:general_cvg_bF} are satisfied despite the fact that $\H$ is not convex in this case.

\section{Fenchel--Moreau identity}\label{section:biconjugation}

The goal is to prove the following version of the Fenchel--Moreau identity on $\S^K_+$. More general versions on self-dual cones in possibly infinite dimensional Hilbert spaces can be seen in \cite{chen2020fenchel}. Here, for completeness, we prove this using arguments more specific to matrices. Recall the Fenchel transformation over $\S^K_+$ defined in \eqref{eq:def_u*}, and the sense of nondecreasingness in \eqref{eq:nondecreasing}.

\begin{Prop}[Fenchel--Moreau identity]\label{prop:biconjugate}
Let $u:\S^K_+\to (-\infty,+\infty]$ be a function not identically equal to $+\infty$. Then, $u^{**}=u$ if and only if $u$ is convex, l.s.c.\ (lower semi-continuous), and nondecreasing.
\end{Prop}

It is easy to see that $v^*$ is convex and l.s.c.\ for any function $v$. In addition by Lemma~\ref{lemma:psd}, we can see that $v^*$ is also nondecreasing. Hence, to prove Proposition~\ref{prop:biconjugate}, it suffices to show the following.

\begin{Lemma}\label{prop:biconjugate_one_side}
If $u:\S^K_+\to(-\infty,+\infty]$ is convex, l.s.c., nondecreasing and not identically $+\infty$, then $u^{**}=u$.
\end{Lemma}

The rest of this section is devoted to proving Lemma~\ref{prop:biconjugate_one_side}. Henceforth, we assume that $u$ satisfies the condition imposed in this lemma.

\subsection{Preliminaries}
We introduce some notation and classical results.  We extend $u$ to $\S^K \cong \R^{K(K+1)/2} $ by setting the value outside $\S^K_+$ to be $\infty$. Denote by $\circledast$ the usual conjugate with the $\sup$ over $\S^K$. The extension of $u$ gives $u^\circledast = u^*$. By the regular Fenchel-Moreau theorem, we have
\begin{align*}
u(x)= \sup_{y\in\S^K}\{y\cdot x - u^*(y)\},\quad \forall x\in\S^K.
\end{align*}
We want to show, whenever $x\in\S^K_+$, the $\sup$ above can be taken over $\S^K_+$.

\smallskip

Denote by $\Omega = \mathsf{dom}\, u=\{x\in\S^K:\ u(x)<+\infty\}$ the effective domain of $u$.  For any $A\subset \S^K$, $\itr A$, $\cl A$, $\bd A$ and $\conv$ stand for the interior, closure, boundary, and convex hull of $A$, respectively. For each $y\in\S^K$, we define the subdifferential of $u$ at $x$ by
\begin{align}\label{eq:def_subdiff}
    \partial u(y) = \{z\in\S^K:\ u(y') \geq u(y)+z\cdot(y'-y),\ \forall y'\in \S^K\}.
\end{align}
The outer normal cone to $\Omega$ at $y\in\S^K$ is given by
\begin{align}\label{eq:def_normal_cone}
    \nn(y)=\{z\in\S^K:\ z\cdot(y'-y)\leq 0,\ \forall y'\in \Omega\}.
\end{align}
Define
\begin{align*}
    D = \{x\in\Omega:\ u \text{ is differentiable at }x\}. 
\end{align*}
For $a\in\S^K$ and $\nu\in\R$, we define the affine function $L_{a,\nu}$ by $L_{a,\nu}(x)=a\cdot x + \nu$.

\smallskip

We recall some useful lemmas, all of which are classical.

\begin{Lemma}\label{lemma:interior}
For a convex set $A$, if $y\in \cl A$ and $y'\in\itr A$, then $\lambda y + (1-\lambda)y'\in\itr A$ for all $\lambda \in [0,1)$.
\end{Lemma}

\begin{Lemma}\label{lemma:cty}
Let $x\in\S^K_+$ and $y\in\Omega$. For every $\alpha\in(0,1)$, set $x_\alpha=(1-\alpha)x+\alpha y$. Then $\lim_{\alpha\to 0} u(x_\alpha)=u(x)$.
\end{Lemma}

\begin{Lemma}\label{lemma:diff_D}
The set $\itr \Omega\setminus D$ has Lebesgue measure zero.
\end{Lemma}

\begin{Lemma}\label{lemma:subdiff}
If $\itr\Omega\neq \emptyset$, then
\begin{align*}
    \partial u(y)=\cl \big(\conv A(y)\big)+\nn(y),\quad\forall y\in\Omega,
\end{align*}
where $A(y)$ is the set of all limits of sequences $\big(\nabla u(y_n)\big)_{n=1}^\infty$ with $\lim_{n\to\infty}y_n =y$ and $y_n\in D$ for all $n$.
\end{Lemma}

\begin{Lemma}\label{lemma:grad_S^K_+}
If $\partial u(y)\cap \S^K_+\neq \emptyset$, then $u^{**}(y)=u(y)$.
\end{Lemma}
\begin{Lemma}\label{lemma:supporting_hyperplane}
For every $x\in\S^K$, we have $u^{**}(x)=\sup L_{a,\nu}(x)$, where the supremum is taken over the set $\{L_{a,\nu}:\ a\in\S^K_+,\ \nu\in\R,\ L_{a,\nu}\leq u\}$.
\end{Lemma}
Lemma~\ref{lemma:interior},  \ref{lemma:subdiff}, and \ref{lemma:grad_S^K_+} can be derived from \cite[Theorem 6.1, 25.6, and 23.5 ]{rockafellar1970convex}, respectively. Lemma~\ref{lemma:cty} is borrowed from \cite[Proposition 9.14]{bauschke2011convex}. The density claim in Lemma~\ref{lemma:diff_D} follows from \cite[Theorem 25.5]{rockafellar1970convex}. The idea to verify the boundedness assertion can be seen in the proof of \cite[Proposition 6.2.2 in Chapter D]{hiriart2012fundamentals}. Lemma~\ref{lemma:supporting_hyperplane} can be verified using the definitions of $u^{**}$ and $\sup L_{a,\nu}(\cdot)$.

\smallskip

In Section~\ref{section:nonempty_int}, we prove Lemma~\ref{prop:biconjugate_one_side} under an additional assumption that $\itr\Omega\neq \emptyset$. We consider the case $\itr\Omega=\emptyset$ in Section~\ref{section:empty_int}.

\subsection{Case 1: nonempty interior}\label{section:nonempty_int}
Assuming $\itr\Omega \neq \emptyset$, we want to show that the identity $u^{**}=u$ holds for all $x\in\S^K_+$. We proceed in steps and show this identity holds on $\cl\Omega$ and then on $\S^K_+$.

\subsubsection{Analysis on $\cl \Omega$} 
At every $x\in D$, due to the nondecreasingness of $u$, we have $a\cdot\nabla u(x)\geq 0$ for all $a\in\S^K_+$. Then, Lemma~\ref{lemma:psd} implies $\nabla u(x)\in\S^K_+$ at every $x\in D$. By Lemma~\ref{lemma:grad_S^K_+}, we conclude $u^{**}(x)=u(x)$ for all $x\in D$.

\smallskip

Now for each $x\in\cl\Omega$, since $\itr\Omega$ is convex and nonempty, by Lemma~\ref{lemma:diff_D} and an argument using Fubini's theorem, we can see that there is $x'\in\Omega$ such that $x_\alpha = (1-\alpha)x+\alpha x' \in D $ for every $\alpha\in(0,1)$. Since both $u^{**}$ and $u$ are convex and l.s.c., Lemma~\ref{lemma:cty} implies that $u^{**}(x)=u(x)$ for all $x\in\cl\Omega$.

\subsubsection{Analysis on $\S^K_+$}

Let $x \in \S^K_+\setminus \cl \Omega$. Hence, we have $u(x)=\infty$. Define
\begin{align*}
    \lambda'  = \sup\{\lambda\in[0,+\infty):\ \lambda x \in \cl\Omega\}.
\end{align*}
By $x\not\in\cl\Omega$, $0\in\cl\Omega$ and the convexity of $\cl\Omega$, we must have
\begin{align}\label{eq:lambda'<1}
    \lambda'<1.
\end{align}
Set $x'=\lambda' x$. It is clear that $x'\in\cl\Omega$ and satisfies \eqref{eq:case_2}. The definition of $\lambda'$ also ensures $x'\not\in\itr\Omega$. There are two cases, either $x'\in\Omega$ or not.

\smallskip

When $x'\notin\Omega$, by $u^{**}=u$ on $\cl\Omega$ and Lemma~\ref{lemma:supporting_hyperplane},
there is a sequence of affine functions $(L_{a_n,\nu_n})_{n=1}^\infty$ such that $a_n\in\S^K_+$, $u\geq L_{a_n,\nu_n}$ for all $n$ and $u(x')=\lim_{n\to\infty}L_{a_n,\nu_n}(x')=\infty$. By the definition of $x'$ and \eqref{eq:lambda'<1}, we can see \begin{align*}
    L_{a_n,\nu_n}(x)=L_{a_n,\nu_n}(x')+(1-\lambda')a_n\cdot x\geq L_{a_n,\nu_n}(x').
\end{align*}
Hence, we also have $u(x)=\lim_{n\to\infty}L_{a_n,\nu_n}(x)=\infty$. This together with Lemma~\ref{lemma:supporting_hyperplane} shows $u^{**}=u$ at this $x$.

\smallskip

Now, we turn to the case where $x'\in\Omega$. We need the next lemma.
\begin{Lemma}\label{lemma:outer_normal_x}
For every $x\in\bd\Omega$ satisfying
\begin{align}
    \lambda x \not\in \cl \quad\forall\lambda >1,\label{eq:case_2}
\end{align}
we have $\big(\nn(x)\cap\S^K_+\big)\setminus\{0\}\neq \emptyset$.
\end{Lemma}
Since $x'$ satisfies \eqref{eq:case_2}, this lemma implies that there is $z\in\nn(x')\cap \S^K_+$ with $z\neq 0$.
The definition \eqref{eq:def_normal_cone} yields
\begin{align}\label{eq:z_n(x')}
    z\cdot(y-x')\leq 0,\quad \forall y\in\Omega.
\end{align}
Since we clearly have $0\in\Omega$, we have $z\cdot x'\geq 0$. We claim that actually
\begin{align}\label{eq:zx'>0}
    z\cdot x'>0.
\end{align}
Otherwise, we have  $z\cdot x'=0$. Since there is $x_0 \in \itr\Omega\subset \S^K_{++} $, we can see that there is $\eps>0$ sufficiently small such that $x_0-\eps z \in \S^K_+$. The nondecreasingness of $u$ yields $\eps z\in\Omega$. Replacing $y$ by $\eps z$ in \eqref{eq:z_n(x')} and using $z\cdot x'=0$, we have $\eps |z|^2\leq 0$, contradicting $z\neq 0$. Hence, we have \eqref{eq:zx'>0}.

\smallskip

By $u^{**}=u$ on $\cl\Omega$ and Lemma~\ref{lemma:supporting_hyperplane}, we can find an affine function $L_{a,\nu}$ with $a\in\S^K_+$ such that $u\geq L_{a,\nu}$. Now, for each $\rho\geq 0$, define
\begin{align*}
    \cL_\rho = L_{a+\rho z,\ \nu-\rho z\cdot x'}.
\end{align*}
Due to \eqref{eq:z_n(x')}, we can see
\begin{align*}
    \cL_\rho(y)= L_{a,\nu}(y)+\rho z\cdot (y-x')\leq L_{a,\nu}(y)\leq u(y),\quad \forall y \in\Omega.
\end{align*}
Since $u=\infty$ outside $\Omega$, we thus have $\cL_\rho\leq u$. On the other hand, we can compute
\begin{align*}
    \cL_\rho(x)=L_{a,\nu}(x)+\rho z\cdot(x-x')= L_{a,\nu}(x) + \rho(\lambda'^{-1}-1)z\cdot x'.
\end{align*}
By \eqref{eq:lambda'<1} and \eqref{eq:zx'>0}, we have $\lim_{\rho\to\infty}\cL_\rho(x)=\infty=u(x)$. By Lemma~\ref{lemma:supporting_hyperplane}, this shows that $u^{**}=u$ holds at $x\in\S^K_+\setminus\cl\Omega$. Together with previous results, we conclude that $u^{**}=u$ holds on $\S^K_+$ under the assumption $\itr\Omega \neq \emptyset$.

\smallskip

To complete the proof of Lemma~\ref{prop:biconjugate_one_side} under the additional assumption $\itr\Omega\neq \emptyset$, it remains to prove Lemma~\ref{lemma:outer_normal_x}.

\begin{proof}[Proof of Lemma~\ref{lemma:outer_normal_x}]

Fix $x\in\Omega\setminus\itr\Omega$ satisfying \eqref{eq:case_2}.

\smallskip

Step 1. We show that for every Euclidean ball $B\subset \S^K$ centered at $x$, there is $\bar x\in \S^K_{++}\cap \bd \Omega$.
By \eqref{eq:case_2}, there is some $\lambda>1$ such that $x' = \lambda x \in B\setminus \cl\Omega$. By $\itr\Omega \neq \emptyset$ and Lemma~\ref{lemma:interior}, there is $x'' \in B\cap \itr \Omega\subset \S^K_{++}$. For $\rho\in[0,1]$, set
\begin{align*}
    x_\rho = \rho x' + (1-\rho)x''\in B.
\end{align*}
Set $\rho_0=\sup \{\rho\in[0,1]:\ x_\rho\in\itr\Omega\}$. We can see $x_{\rho_0}$ lies in the closure but not the interior of $\Omega$, and thus $x_{\rho_0}\in B\cap \bd \Omega $. In addition, since $x'\not \in \cl \Omega$, we must have $\rho_0<1$ and hence $x_{\rho_0}\in\S^K_{++}$ due to $x''\in\S^K_{++}$. We conclude that $x_{\rho_0}\in B\cap\S^K_{++}\cap\bd \Omega$ is the point $\bar x$ we want.

\smallskip

Step 2. By the construction above, we can find a sequence $(x_n)_{n=1}^\infty$ such that $x_n \in \S^K_{++}\cap \bd \Omega$ and $\lim_{n\to\infty}x_n =x$. 
We want to show $\nn(x_n)\subset \S^K_+$ using the following lemma.

\begin{Lemma}
If $y\in  \S^K_{++}\cap \bd \Omega$, then $\nn(y)\subset \S^K_+$.
\end{Lemma}

\begin{proof}
Since $y\in\bd\Omega$, using $\itr\Omega \neq \emptyset$ and Lemma~\ref{lemma:interior}, we can find $y_\eps \in\itr\Omega$ such that $|y_\eps - y|<\eps$ for each $\eps>0$. By this and $y\in \S^K_{++}$, there are $\eps_0,\delta_0>0$ such that
\begin{align*}
    y_\eps-\delta_0I\in\S^K_+
,\quad \eps\in(0,\eps_0).
\end{align*}
This further implies that there is $\delta>0$ such that
\begin{align*}
    y_\eps -a \in \S^K_+,\quad \forall \eps\in(0,\eps_0)\quad, \forall a\in\S^K_+ \text{ satisfying } |a|\leq \delta.
\end{align*}
Since $u$ is nondecreasing and $y_\eps\in\Omega$, we have $y_\eps -a\in\Omega$ for any $a$ described above. Let $z\in\nn(y)$. The definition \eqref{eq:def_normal_cone} yields $z\cdot(y_\eps -a -y)\leq 0$ and thus
\begin{align*}
    z\cdot a \geq -|z|\eps.
\end{align*}
Sending $\eps\to0$ and varying $a$, we conclude using Lemma~\ref{lemma:psd} that $z\in \S^K_+$.

\end{proof}

This lemma immediately implies that $\nn(x_n)\subset \S^K_+$. For each $n$, pick $z_n \in \nn(x_n)\cap \S^K_+$ with $|z_n|=1$. By extracting a subsequence, we may assume $\lim_{n\to\infty}z_n = z$ for some $z\in\S^K_+$ satisfying $|z|=1$. Since $z_n\in\nn(x_n)$, we have
\begin{align*}
    z_n\cdot(y-x_n)\leq 0,\quad \forall y\in\S^K.
\end{align*}
Set $n\to\infty$, recall that $\lim_{n\to\infty}x_n =x$, and we obtain $z\cdot(y-x)\leq 0$ for all $y\in\S^K$. 
This proves Lemma~\ref{lemma:outer_normal_x}.

\end{proof}

\subsection{Case 2: empty interior}\label{section:empty_int}
To complete the proof of Lemma~\ref{prop:biconjugate_one_side}, let us investigate the situation where $\itr\Omega = \emptyset$.
The case $\Omega = \{0\}$ is easy to handle. So, we assume $\itr\Omega =\emptyset$ and $\Omega\setminus\{0\}\neq \emptyset$. Set
\begin{align}\label{eq:def_J_rank}
    J=\max \{\rank(x):\ x\in\Omega\},
\end{align}
where $\rank(x)$ is the rank of the matrix $x$. By $\Omega\setminus\{0\}\neq \emptyset$, we have $J\geq 1$.

\medskip

Step 1. We show $J<K$. Otherwise, there is $x\in\Omega$ with $\rank(x)=K$. Hence, we have $x\in \S^K_{++}$. Therefore, there is $\delta>0$ such that $x-y\in\S^K_{++}$, for all $y\in\S^K_+$ with $|y|\leq \delta$. This contradicts the assumption that $\itr\Omega = \emptyset$.

\smallskip 

For each $n\in\N$, we denote the $n\times n$ zero matrix by $\zero_n$. Fix any $x\in\Omega$ with $\rank(x)=J$. Without loss of generality, by an orthogonal transformation, we may assume $x = \diag(\lambda_1,\lambda_2,\dots,\lambda_J,\zero_{K-J})$, where $\lambda_j>0$ for all $1\leq j\leq J$.

\medskip

Step 2. We show that for every $y\in\Omega$, there is $y^\circ \in \S^J_+$ such that
\begin{align}\label{eq:y_diag_form}
    y=\diag(y^\circ, \zero_{K-J}).
\end{align}
Otherwise, there is $y\in \Omega$ with $y_{ij}\neq 0$ for some $i>J$ or $j>J$. Since $y\in\S^K_+$ is positive semidefinite, we must have $y_{ii}>0$ for some $i>J$. By reordering, we assume $i=J+1$. Note that this reordering preserves $x$. We want to show $\rank(x+y)>J$. Let $\hat y = (y_{ij})_{1\leq i,j\leq J+1}\in\S^{J+1}_+$ be a portion of $y$, and $\hat x$ be similarly defined. It suffices to show $\rank(\hat x + \hat y) = J+1$. We further reduce this to verifying $\hat x + \hat y\in \S^{J+1}_{++}$ and thus showing
\begin{align}\label{eq:v^T(x+y)v>0}
    v^\intercal (\hat x +\hat y)v >0
\end{align}
for all $v\in\R^{J+1}\setminus \{0\}$.

\smallskip

First, we consider the case where $v_j\neq 0$ for some $1\leq j\leq J$. Since $\hat x = \diag(\lambda_1,\dots,\lambda_J,0)$ and each $\lambda_j$ is positive, we have $v^\intercal \hat x v = \sum_{j=1}^J\lambda_jv_j^2>0$, verifying \eqref{eq:v^T(x+y)v>0}. Now, suppose $v_j=0$ for all $1\leq j\leq J$. Due to $v\in\R^{J+1}\setminus \{0\}$, we must have $v_{J+1}\neq 0$. Since $y_{J+1,J+1}>0$, we obtain $v^\intercal \hat y v= y_{J+1,J+1}v_{J+1}^2>0$. In conclusion, \eqref{eq:v^T(x+y)v>0} holds.

\smallskip

Therefore, $\rank(\hat x +\hat y)=J+1$, and thus $\rank(x+y)>J$. By the convexity of $\Omega$, we see that $\frac{1}{2}(x+y)\in \Omega$. But this contradicts \eqref{eq:def_J_rank}. Hence, by contradiction, $y$ is of the form \eqref{eq:y_diag_form} for all $y\in\Omega$.

\medskip

Step 3. We apply the result in the previous section. Define
\begin{align*}
    \cC = \{\diag(y^\circ,\zero_{K-J}):\ y^\circ \in \S^J_+\}\subset \S^K_+.
\end{align*}
By the result from Step 2, we have $\Omega \subset \cC$. Identifying $\cC$ with $\S^J_+$, we can view $u$ as a map from $\S^J_+$ to $(-\infty,\infty]$. By \eqref{eq:def_J_rank}, the interior of $\Omega$ relative to $\S^J_+$ is nonempty. Hence, applying the result for case with nonempty interior, comparing with $u^{**}=u$, we have
\begin{align}\label{eq:holds_on_cC}
    u(x) = \sup_{z\in\cC}\{z\cdot x - u^*(z)\},\quad \forall x\in\cC.
\end{align}
Since $u\geq u^{**}$, we have $u^{**}=u$ on $\cC$.

\medskip

Step 4. To complete the proof, we show that $u^{**}=u$ holds on $\S^K_+\setminus\cC$. Let us set $z=\diag\{\zero_{J},I_{K-J}\}$ where $I_{K-J}$ is the $(K-J)\times(K-J)$ identity matrix. Fix any $x\in\S^K_+\setminus \cC$. Due to $x\not\in\cC$, there is some $i>J$ or $j>J$ such that $x_{ij}\neq 0$. Since $x$ is positive semidefinite, we must have $x_{ii}>0$ for some $i>J$. Therefore, we get
\begin{align}\label{eq:z_cdot_x>0}
    z\cdot x>0.
\end{align}
By \eqref{eq:holds_on_cC}, there is an affine function $L_{a,\nu}$ with $a\in\cC\subset\S^K_+$ such that $u\geq L_{a,\nu}$ on $\cC$. Now, for every $\rho\geq 0$, we define
\begin{align*}
    \cL_\rho = L_{a+\rho z,\nu}.
\end{align*}
By the definition of $z$, we can compute
\begin{align*}
    \cL_\rho(y) = L_{a,\nu}(y)+z\cdot y=L_{a,\nu}(y)\leq u(y),\quad \forall y\in\cC.
\end{align*}
Since $u=\infty$ outside $\cC$, we then get $\cL_\rho\leq u$. On the other hand, \eqref{eq:z_cdot_x>0} implies that
\begin{align*}
    \cL_\rho(x)= L_{a,\nu}(y)+\rho z\cdot y
\end{align*}
converges to $\infty$ as $\rho\to\infty$. Then, Lemma~\ref{lemma:grad_S^K_+} implies $u^{**}=u$ at $x\in\S^K_+\setminus \cC$.

\section{Concentration in the special case}

In this appendix, we prove a concentration result assuming $X$ has i.i.d.\ and bounded entries. The following lemma works for any fixed interaction matrix $A\in\R^{K^p\times L}$ in \eqref{eq:observ}. Recall the definition of $\cK_{M,N}$ in \eqref{eq:def_K}.

\begin{Lemma}\label{lemma:concentration}
Assume that $X$ consists of i.i.d.\ entries and $|X_{ij}|\leq 1$ for all $i$ and $j$. Then, there is $C>0$ such that the following holds for all $M\geq1$ and $n\in \N$,
\begin{align*}
    \cK_{M,N} \leq CN^{-\frac{1}{2}}\big( M+\sqrt{\log N}\big).
\end{align*}
\end{Lemma}

\subsection{Proof of Lemma~\ref{lemma:concentration}}

The plan is to first obtain an  estimate of $\E e^{\lambda^2 N|F_N-\bF_N|^2}$ for small $\lambda>0$ pointwise at each $(t,h)\in [0,M]\times \SM$. Then, we use an $\eps$-net argument to bound $\E \sup_{(t,h)\in[0,M]\times \SM} e^{\lambda^2 N |F_N-\bF_N|}$. The desired result follows from Jensen's inequality.

\subsubsection{Pointwise estimate}

Let $(t,h)\in [0,M]\times \SM$. 
Denote by $G=(W,Z)$ the Gaussian vector consisting of all Gaussian random variables in $F_N$. We also write $\E_G$, $\E_X$ as the expectation integrating over $G$, $X$, respectively. Let $\lambda>0$ be chosen later. Using the Cauchy--Schwarz inequality, we have
\begin{align}
    \E e^{\lambda |F_N-\bF_N|} &\leq \E\Big( e^{\lambda |F_N-\E_XF_N|}e^{\lambda |\E_X F_N-\E_{X,G}F_N|}\Big)\nonumber\\
    &= \Big(\E e^{2\lambda |F_N-\E_XF_N|}\Big)^\frac{1}{2}\Big(\E e^{2\lambda |\E_X F_N-\E_{X,G}F_N|}\Big)^\frac{1}{2}.\label{eq:exp_moment_F_N}
\end{align}

To treat the last term, we will use the Gaussian concentration inequality. 
Let us use the multi-index notation \eqref{eq:multi-index}. By \eqref{eq:H_N_enriched} and \eqref{eq:def_F_N}, we can compute
\begin{align*}
    \partial_{W_{\ii}}F_N = \frac{1}{N}\sqrt{\frac{2t}{N^{p-1}}}\la \tx_\ii\ra ,\quad 
    \partial_{Z_{ij}}F_N = \frac{1}{N}\sum_{k=1}^K \big(\sqrt{2h}\big)_{kj}\la x_{ik}\ra.
\end{align*}
Here $\tx$ is defined in \eqref{eq:tx}. Therefore, by \eqref{eq:support}, we have
\begin{align*}
    |\nabla_G F_N|^2& = \sum_{\ii}|\partial_{W_{\ii}}F_N|^2 + \sum_{i=1}^N\sum_{j=1}^K|\partial_{Z_{ij}}F_N|^2\\
    &=\frac{2t}{N^{p+1}}\la \tx\cdot \tx'\ra  + \frac{2}{N^2}h\cdot \la x^\intercal x'\ra \leq CMN^{-1}.
\end{align*}
Invoking \cite[Theorem 5.5]{boucheron2013concentration}, we obtain
\begin{align}\label{eq:G_exp}
    \E_G e^{\lambda |\E_{X} F_N-\E_{X,G}F_N|} \leq e^{C\lambda^2MN^{-1}}.
\end{align}

Then, we treat the first two terms in \eqref{eq:exp_moment_F_N}. Let us first compute $\partial_{X_{ij}}F_N$. By \eqref{eq:def_F_N}, we can compute
\begin{align*}
    \partial_{X_{ij}}F_N &= \frac{1}{N}\bigg\la \frac{2t}{N^{p}}\partial_{X_{ij}}\big(\tx \cdot \tilde X\big) +2 \partial_{X_{ij}}\Big(h\cdot \big(x^\intercal X\big)\Big) \bigg\ra\nonumber.
\end{align*}
Due to the boundedness assumption $|X_{i,\cdot}|\leq \sqrt{K}$ (and thus $|x_{i,\cdot}|\leq \sqrt{K}$ under the distribution $\la\, \cdot\, \ra$), we can verify 
\begin{align*}
    \big|\nabla_{X_{i,\cdot}}F_N\big| \leq CMN^{-1}.
\end{align*}
Using the boundedness again and \cite[Theorem 6.2]{boucheron2013concentration} (see the penultimate display in its proof), we obtain
\begin{align}\label{eq:X_exp}
\E_{X}e^{\lambda |F_N - \E_{X} F_N|}\leq Ce^{C\lambda^2M^2N^{-1}}.
\end{align}

\smallskip

In conclusion, \eqref{eq:exp_moment_F_N}, \eqref{eq:G_exp} and \eqref{eq:X_exp}, with $\lambda$ replaced by $\lambda \sqrt{N}$, yield
\begin{align*}
    \E e^{\lambda \sqrt{N}|F_N-\bF_N|} \leq Ce^{C\lambda^2M^2}.
\end{align*}
Then, \cite[Proposition 2.5.2]{vershynin2018high} implies that, for $\lambda$ sufficiently small,
\begin{align}\label{eq:F_N_ptw_concentration}
    \E e^{\lambda^2 N|F_N-\bF_N|^2}\leq Ce^{C\lambda^2 M^2}.
\end{align}

\subsubsection{Application of an $\eps$-net argument}

The goal is to upgrade \eqref{eq:F_N_ptw_concentration} to a bound on  $\E\sup_{(t,h)\in[0,M]\times\SM}e^{\lambda^2N|F_N-\bF_N|^2}$. The estimates \eqref{eq:1st_der_bF} and \eqref{eq:1st_der_F} imply that, for $|t-t'|+|h-h'|\leq 1$,
\begin{align*}
    |F_N(t,h)- F_N(t',h')|\leq C\Big(1+N^{-\frac{1}{2}}\big(\|WA^\intercal\|+|Z|\big)\Big)\big(|t-t'|^\frac{1}{2}+|h-h'|^\frac{1}{2}\big).
\end{align*}
For $\eps\in(0,1]$, viewing $\SM$ as a subset of $\R^{K(K+1)/2}$, we introduce the $\eps $-net
\begin{align*}
    A_\eps = \{\eps,2\eps, 3\eps\dots\}^{1+K(K+1)/2}\cap \Big([0,M]\times \SM\Big).
\end{align*}
Hence, for $\lambda$ small, we have
\begin{align}
    &\E\sup_{(t,h)\in[0,M]\times \SM}e^{\lambda^2 N|F_N-\bF_N|^2}\nonumber\\
    &\leq \E \exp \Big( C\lambda^2\eps\big(\sqrt{N}+\|WA^\intercal\|+|Z|\big)^2\Big)\sup_{(t,h)\in A_\eps} e^{\lambda^2N|F_N-\bF_N|^2}\nonumber\\
    &\leq\bigg(\E \exp\Big( C\lambda^2 \eps\big(\sqrt{N}+\|WA^\intercal\|+|Z|\big)^2\Big)\bigg)^\frac{1}{2} \bigg(\E\sup_{(t,h)\in A_\eps}e^{2\lambda^2N|F_N-\bF_N|^2}\bigg)^\frac{1}{2}\label{eq:exp_sup_F_N}
\end{align}
where we used the Cauchy--Schwarz inequality in the second inequality.
Since $|A_\eps|\leq (M/\eps)^{1+K(K+1)/2}$,  using the union bound and \eqref{eq:F_N_ptw_concentration}, we have,
\begin{align}\label{eq:exp_sup_A_eps_F_N}
    \bigg(\E\sup_{(t,h)\in A_\eps}e^{2\lambda^2N|F_N-\bF_N|^2}\bigg)^\frac{1}{2} \leq C(M/\eps)^Ce^{C\lambda^2 M^2}, \quad \lambda \in \R.
\end{align}

Set $\eps = C^{-1}N^{-1}$ in \eqref{eq:exp_sup_F_N} with $C$ therein, and use \eqref{eq:exp_sup_A_eps_F_N} to see
\begin{align*}
\begin{split}
    &\E\sup_{(t,h)\in[0,M]\times\SM}e^{\lambda^2 N|F_N-\bF_N|^2}\\
    &\leq C (MN)^C e^{C\lambda^2 M^2} \bigg[\E \exp\Big(\lambda^2\big(1+N^{-\frac{1}{2}}(\|WA^\intercal\|+|Z|)\big)^2\Big)\bigg]^\frac{1}{2}.
\end{split}
\end{align*}
We claim that, for small $\lambda>0$, 
\begin{align}\label{eq:claim_concentration}
    \E \exp\Big( \lambda^2\big(1+N^{-\frac{1}{2}}(\|WA^\intercal\|+|Z|)\big)^2\Big) \leq C.
\end{align}
This immediately gives
\begin{align*}
    \E\sup_{(t,h)\in[0,M]\times\SM}e^{\lambda^2 n |F_N-\bF_N|^2}\leq C(MN)^Ce^{C\lambda^2 M^2} .
\end{align*}
Finally, using Jensen's inequality, we conclude that
\begin{align*}
    \E\sup_{(t,h)\in[0,M]\times\SM}|F_N-\bF_N|^2&\leq \lambda^{-2}N^{-1}\log\bigg( \E \sup_{(t,h)\in[0,M]\times\SM}e^{\lambda^2N|F_N - \bF_N|^2}\bigg)\\
    &\leq CN^{-1}(M^2 + \log N),
\end{align*}
as desired. The proof will be complete once \eqref{eq:claim_concentration} is verified.

\subsubsection{Proof of \eqref{eq:claim_concentration}}

We want to bound exponential moments of $\|WA^\intercal\|^2$ and $|Z|^2$. Using the fact that $Z$ is standard Gaussian in $\R^N$, we have, for $\lambda$ small,
\begin{align}\label{eq:exp_|Z|}
    \E e^{\lambda^2 N^{-1}|Z|^2} \leq C.
\end{align}

\smallskip

Now, we turn to bound $\E e^{\lambda^2 N^{-1}\|WA^\intercal\|^2}$. 
For each $\eps>0$, there is a finite set $B\subset \mathbb{S}^{NK-1}$ such that for each $y\in \mathbb{S}^{NK-1}$ there is $z\in B$ satisfying $|y-z|\leq \eps$. In addition, the size of $B$ is bounded by $a^{NK}$ for some constant $a>0$ depending only on $\eps$. The construction of $B$ is classical and can be seen, for instance, in \cite[Corollary 4.2.13]{vershynin2018high}. Using the property of $B$, we can see that for each $(y_1,\, y_2,\, \dots\, y_p)\in (\mathbb{S}^{NK-1})^p$ there is $(z_1,\, z_2,\, \dots\, z_p)\in B^p$ such that
\begin{align*}
    \Big|(WA^\intercal)\cdot (y_1\otimes y_2\otimes \cdots \otimes y_p) - (WA^\intercal)\cdot (z_1\otimes z_2\otimes \cdots \otimes z_p)\Big|\leq p\eps \|WA^\intercal\|.
\end{align*}
By this and fixing $\eps = \frac{1}{2p}$, from the definition \eqref{eq:||w||}, we obtain
\begin{align*}
    \|WA^\intercal \|\leq 2\sup_{(z_1,\, z_2,\, \dots\, z_p)\in B^p}(WA^\intercal)\cdot (z_1\otimes z_2\otimes \cdots \otimes z_p).
\end{align*}
Note that $(WA^\intercal)\cdot (z_1\otimes z_2\otimes \cdots \otimes z_p)$ is a centered Gaussian with variance bounded by a constant $C$ depending only on $A$. Therefore, there is $\gamma>0$ such that
\begin{align*}
    \P\Big\{(WA^\intercal)\cdot (z_1\otimes z_2\otimes \cdots \otimes z_p)\geq t\Big\}\leq 2e^{-\gamma t^2 }.
\end{align*}
Combine the above two displays and apply the union bound to see
\begin{align*}
    \P\big\{e^{\lambda^2N^{-1}\|WA^\intercal \|^2}\geq t\big\}\leq 2\bigg(\frac{a^{pK}}{t^c}\bigg)^N
\end{align*}
for some constant $c>0$ that absorbs $\lambda$ and $\gamma$. Writing $b=a^\frac{pK}{c}$, we have, for $N$ large,
\begin{align*}
    \E e^{\lambda^2 N^{-1}\|WA^\intercal\|^2} = \int_0^\infty \P\{e^{\lambda^2 N^{-1}\|WA^\intercal\|^2}\geq t\}\d t\leq b + \int_b^\infty 2\bigg(\frac{b}{t}\bigg)^{cN}\d t=b+\frac{2b}{cN-1},
\end{align*}
which is bounded uniformly for large $N$.
This and \eqref{eq:exp_|Z|} imply \eqref{eq:claim_concentration}.

\newcommand{\noop}[1]{} \def\cprime{$'$}

\bibliographystyle{abbrv}

\end{document}